\documentclass[reqno]{amsart}

\address{Department of Mathematics, Kyoto
University, Kyoto, Japan } \email{fukaya@math.kyoto-u.ac.jp}
\address{Department of Mathematics, University of
Wisconsin, Madison, WI, USA \& \newline \indent Department of Mathematics,
POSTSECH, Pohang, Korea } \email{oh@math.wisc.edu}
\address{Graduate School of Mathematics,
Nagoya University, Nagoya, Japan \&
Korea Institute for Advanced Study, Seoul, Korea,} \email{ohta@math.nagoya-u.ac.jp}
\address{Department of Mathematics,
Hokkaido University, Sapporo, Japan \&
Korea Institute for Advanced Study, Seoul, Korea,}
\email{ono@math.sci.hokudai.ac.jp}

\usepackage[all]{xy}
\usepackage{graphicx}

\usepackage{amssymb}
\usepackage{amsmath}
\usepackage{mathrsfs}
\usepackage{epsfig}
\usepackage{amscd}
\usepackage{graphicx, color}

\def\E{\ifmmode{\mathbb E}\else{$\mathbb E$}\fi} 
\def\N{\ifmmode{\mathbb N}\else{$\mathbb N$}\fi} 
\def\R{\ifmmode{\mathbb R}\else{$\mathbb R$}\fi} 
\def\Q{\ifmmode{\mathbb Q}\else{$\mathbb Q$}\fi} 
\def\C{\ifmmode{\mathbb C}\else{$\mathbb C$}\fi} 
\def\H{\ifmmode{\mathbb H}\else{$\mathbb H$}\fi} 
\def\Z{\ifmmode{\mathbb Z}\else{$\mathbb Z$}\fi} 
\def\P{\ifmmode{\mathbb P}\else{$\mathbb P$}\fi} 
\def\T{\ifmmode{\mathbb T}\else{$\mathbb T$}\fi} 
\def\SS{\ifmmode{\mathbb S}\else{$\mathbb S$}\fi} 
\def\DD{\ifmmode{\mathbb D}\else{$\mathbb D$}\fi} 

\newcommand{\e}{\varepsilon}

\newcommand{\del}{\partial}

\newcommand{\ben}{\begin{enumerate}}
\newcommand{\een}{\end{enumerate}}
\newcommand{\be}{\begin{equation}}
\newcommand{\ee}{\end{equation}}
\newcommand{\bea}{\begin{eqnarray}}
\newcommand{\eea}{\end{eqnarray}}
\newcommand{\beastar}{\begin{eqnarray*}}
\newcommand{\eeastar}{\end{eqnarray*}}
\newcommand{\bc}{\begin{center}}
\newcommand{\ec}{\end{center}}

\theoremstyle{theorem}
\newtheorem{thm}{Theorem}[section]
\newtheorem{cor}[thm]{Corollary}
\newtheorem{lem}[thm]{Lemma}
\newtheorem{prop}[thm]{Proposition}

\theoremstyle{definition}
\newtheorem{defn}[thm]{Definition}
\newtheorem{rem}[thm]{Remark}

\newtheorem*{thm*}{Theorem}

\numberwithin{equation}{section}

\def\R{{\mathbb R}}

\def\Crit{{\hbox{Crit}}}

\def\E{{\mathbb E}}
\def\Z{{\mathbb Z}}
\def\C{{\mathbb C}}
\def\R{{\mathbb R}}
\def\P{{\mathbb P}}

\def\N{{\mathbb N}}

\def\11{{\mathbb I}}

\def\dudtau{{\frac{\del u}{\del \tau}}}
\def\dudt{{\frac{\del u}{\del t}}}

\def\C{\mathbb{C}}
\def\Z{\mathbb{Z}}

\def\T{\mathbb{T}}

\def\Q{\mathbb{Q}}

\def\E{\ifmmode{\mathbb E}\else{$\mathbb E$}\fi} 
\def\N{\ifmmode{\mathbb N}\else{$\mathbb N$}\fi} 
\def\R{\ifmmode{\mathbb R}\else{$\mathbb R$}\fi} 
\def\Q{\ifmmode{\mathbb Q}\else{$\mathbb Q$}\fi} 
\def\C{\ifmmode{\mathbb C}\else{$\mathbb C$}\fi} 
\def\H{\ifmmode{\mathbb H}\else{$\mathbb H$}\fi} 
\def\Z{\ifmmode{\mathbb Z}\else{$\mathbb Z$}\fi} 
\def\P{\ifmmode{\mathbb P}\else{$\mathbb P$}\fi} 
\def\SS{\ifmmode{\mathbb S}\else{$\mathbb S$}\fi} 
\def\DD{\ifmmode{\mathbb D}\else{$\mathbb D$}\fi} 

\def\R{{\mathbb R}}

\def\Crit{{\hbox{Crit}}}
\def\E{{\mathbb E}}
\def\Z{{\mathbb Z}}
\def\C{{\mathbb C}}
\def\R{{\mathbb R}}

\def\N{{\mathbb N}}
\def\MM{{\mathcal M}}

\def\e{\varepsilon}

\def\CA{{\mathcal A}}

\def\CM{{\mathcal M}}

\def\darr#1{\raise1.5ex\hbox{$\leftrightarrow$}
\mkern-16.5mu #1}

\def\roughly#1{\raise.3ex\hbox{$#1$\kern-.75em
\lower1ex\hbox{$\sim$}}}

\def\opname#1{\mathop{\kern0pt{\rm #1}}\nolimits}

\def\supp{\operatorname{supp}}

\def\leng{\operatorname{leng}}

\def\dist{\operatorname{dist}}
\def\Crit{\operatorname{Crit}}

\begin{document}
\quad \vskip1.375truein

\def\mq{\mathfrak{q}}
\def\mp{\mathfrak{p}}
\def\mH{\mathfrak{H}}
\def\mh{\mathfrak{h}}
\def\ma{\mathfrak{a}}
\def\ms{\mathfrak{s}}
\def\mm{\mathfrak{m}}
\def\mn{\mathfrak{n}}
\def\mz{\mathfrak{z}}
\def\mw{\mathfrak{w}}
\def\Hoch{{\tt Hoch}}
\def\mt{\mathfrak{t}}
\def\ml{\mathfrak{l}}
\def\mT{\mathfrak{T}}
\def\mL{\mathfrak{L}}
\def\mg{\mathfrak{g}}
\def\md{\mathfrak{d}}
\def\mr{\mathfrak{r}}

\title[Displacement of polydisks]{
Displacement of polydisks and Lagrangian Floer theory}

\author[K. Fukaya, Y.-G. Oh, H. Ohta, K.
Ono]{Kenji Fukaya, Yong-Geun Oh, Hiroshi Ohta, Kaoru Ono}
\thanks{KF is supported partially by JSPS Grant-in-Aid for Scientific Research
No.19104001 and Global COE program G08, YO by US NSF grant \# 0904197, HO by JSPS Grant-in-Aid
for Scientific Research Nos. 19340017 and 23340015, and KO by JSPS Grant-in-Aid for
Scientific Research, Nos. 21244002 and 18340014}

\date{April 15, 2011}

\begin{abstract} There are two purposes of the present article.
One is to correct an error in the proof of Theorem 6.1.25 in \cite{fooo:book}, from which
Theorem J \cite{fooo:book} follows. In the course of doing so, we also
obtain a new lower bound of the displacement energy of polydisks in
general dimension. The results of the present article are  motivated by the recent
preprint of Hind \cite{hind} where the 4 dimensional case is studied.
Our proof is different from Hind's even in the
4 dimensional case and provides stronger result, and relies on the study of
torsion thresholds of Floer cohomology of Lagrangian torus fiber in
simple toric manifolds associated to the polydisks.
\end{abstract}

\keywords{Polydisks, Hamiltonian displacement, Lagrangian Floer (co)homology,
torsion threshold, displacement energy}

\maketitle

\tableofcontents

\section{Introduction}
\label{sec:intro}

In \cite{fooo:book}, we stated a
\emph{lower bound} of the displacement energy of
relatively spin Lagrangian submanifold $L$ with a bounding cochain $b$
in terms of torsion exponents of Floer cohomology $HF((L,b),(L,b);\Lambda_{0,\text{\rm nov}})$.
This is Theorem J in \cite{fooo:book}, which is a consequence of Theorem 6.1.25 in the book.
See Theorem \ref{thm:correctedthmJ} and
Theorem \ref{Theorem2} of this paper for the precise statement in a general setting.
However, the proof of Theorem 6.1.25 contains an error.
One of the purposes of the present paper is to correct this error.
There are two key ingredients in this correction of the error, which are interrelated to
each other: one is the usage of an optimal Floer chain map
introduced in the present paper, and an energy estimate of the type
which was first introduced by Chekanov \cite{chekanov} and clarified by
the second named author in \cite{oh:mrl}.

Another purpose is to employ Theorem J \cite{fooo:book} and
obtain an estimate of the lower bound of the displacement energy of
polydisks in cylinder which provides a higher dimensional generalization of
a recent result of Hind \cite{hind} as well as an improvement of Hind's result.

Before we achieve the two purposes above, we prove
the following non-displacement theorem to illustrate a geometric
consequence of Theorem J \cite{fooo:book} in a simple example.
We denote by $S$ the \emph{Gromov width} of any domain of our interest.
For example, the ball $B^{2k} \subset \C^k$ of radius $r > 0$
has Gromov width
$$
S = \pi r^2.
$$
Let $S^1(S) \subset \C$ be a circle of radius $r=\sqrt{\pi^{-1}S}$,
i.e., $S = \pi r^2$.  We also consider $S^1_{eq} \subset S^2(1)$ the
equator in the sphere of area $1$.

\begin{thm}\label{thm:non-hind}
We put $X= \C \times S^2(1)$.
Suppose $S > 1/2$. Consider any time-dependent Hamiltonian $H:[0,1]
\times X \to \R$ with its Hofer norm $\Vert H\Vert < S$. Then we have
$$
\psi_H(S^1(S) \times S^1_{eq}) \cap (S^1(S) \times S^1_{eq}) \ne
\emptyset
$$
for its time-one map $\psi_H:= \phi_H^1$.
\end{thm}

The proof of this theorem is `elementary' in that it uses only the
Lagrangian Floer theory for monotone Lagrangian submanifolds
\cite{oh:floer1} and by now standard computations for the energy
estimates used as in \cite{chekanov}, \cite{oh:mrl}, but does
\emph{not} use any techniques of virtual fundamental chains,
Bott-Morse theory or any higher homological algebra.

In fact, this theorem is a corollary of Theorem J \cite{fooo:book} whose precise
statement we refer to Section \ref{sec:correction}. We provide this elementary
proof in this particular case partly because it nicely illustrates
Theorem J \cite{fooo:book} when a deformation of Floer cohomology by a bounding
cochain is not needed. (On the other hand, the proof of Theorem J \cite{fooo:book}
is given in a very general context in terms of the deformed Floer cohomology of
\emph{weakly unobstructed} Lagrangian submanifolds \emph{after bulk deformations}.)

\begin{rem} Another motivation for us to prove this particular theorem
is related to the \emph{upper bound}  in the following inequality
stated in \cite{hind}:
\be\label{eq:hind} \left(\frac{1}{2}-\e\right)\lfloor S\rfloor + \e
\leq e^{Z_{1,1}}(D(1,S)) \leq \frac{S}{2} + 3 \ee
where $e^{Z_{1,1}}(D(1,S))$ is the displacement energy (Definition \ref{def:dist} (2)) of the bidisk $D(1,S)$ in the cylinder
$Z_{1,1}=Z_{1,1}(1+\e)$.
Here, following the notation of \cite{hind},
we denote the bidisks in $\C^2$ by
$$
D(a,b) = \{(z_1,z_2) \mid \pi |z_1|^2 < a, \,  \pi|z_2|^2 < b\} =
D^2(a) \times D^2(b) \subset \C^2
$$
with $a \leq b$, and
also denote the cylinder in $\C^2$ by
$$
Z_{1,1}(a+\e) = \{(z_1,z_2) \mid \pi |z_1|^2 < a + \e\} = D^2(a+\e) \times \C
$$
for $\e > 0$ small.
Hind's proof of the upper bound uses an explicit construction of displacing Hamiltonian
isotopy. However, the construction used in his proof seems to directly contradict to
the above Theorem \ref{thm:non-hind}, and also to Theorem J \cite{fooo:book}.
In fact, Theorem J \cite{fooo:book} implies the following stronger lower bound
\be\label{eq:fooo}
S \leq e^{Z_{1,1}}(D(1,S))
\ee
whenever $S > 1$.
\end{rem}

Hind \cite{hind} obtained his lower bound in \eqref{eq:hind} by using some embedding obstruction arising from
an explicit study of moduli space of proper holomorphic curves in $S^2 \times S^2 \setminus E$
where $E$ is the image of certain symplectic embedding of ellipsoid. His study on this
lower bound
heavily relies on the compactification result in symplectic field theory \cite{hofer2}, \cite{behwz}.
Partly because such an explicit study of moduli space in high dimension is not
available, Hind restricts himself to the 4 dimensional case. Our proof is different from
Hind's and relies on the study of \emph{torsion thresholds} of Floer cohomology of Lagrangian
torus fiber and Theorem J \cite{fooo:book}.

However, although the statement of Theorem J \cite{fooo:book} is
correct as it is, its proof contains some incorrect argument at the
end of the proof of Theorem 6.1.25 in p.392: \emph{The homomorphisms
in line 9 and 11 of page 392 of \cite{fooo:book} is not well-defined.} And
to give a correct proof of the same statement as stated in Theorem 6.1.25, hence Theorem
J, we also need to improve the energy estimate given in Proposition
5.3.45 \cite{fooo:book} and use a different construction of a Floer
chain map. It turns out that to obtain the the optimal energy
estimate needed to prove Theorem J in the construction of a Floer
chain map, we need to use the Hamiltonian perturbed Cauchy-Riemann
equation with \emph{fixed} Lagrangian boundary condition which
intertwines the so called the geometric version of the Floer
cohomology and the dynamical version of the Floer cohomology and then
applying the \emph{coordinate change} that relates the two. Similar
coordinate change was used by the second named author previously in
\cite{oh:jdg} for a similar purpose. Using this trick and an optimal
energy estimate originated by Chekanov \cite{chekanov}, we can prove
the statement of Theorem 6.1.25 and hence
Theorem J in \cite{fooo:book} as they are  currently stated.

Another purpose of the present paper is to apply Theorem J \cite{fooo:book}
in the study of displacement energy of polydisks in
arbitrary dimension and generalize the above mentioned 4 dimensional result
to arbitrary dimension whose description is now in order.
It turns out that the Lagrangian Floer theory developed in \cite{fooo:book}
and \cite{fooo:toric1,fooo:toric2} can be nicely applied to the various
symplectic topological questions concerning polydisks $D^2(r_1) \times \cdots
\times D^2(r_n)$. This is largely
because the polydisks contain Lagrangian tori
which can be embedded into the toric manifolds
$S^2(a_1) \times \cdots \times S^2(a_n)$ or
$S^2(a) \times \C P^{n-1}(\lambda)$ for suitable choices of $a_i$'s or $(a,\lambda)$.
Here $\C P^{n-1}(\lambda)$ is the projective space with the Fubini-Study K\"ahler form
$\omega$ with $[\omega](C) = \lambda$ for the homology class $C$ of
the complex line. With this notation, we have the
symplectic embedding of $B^{2n}(\lambda) \hookrightarrow \C P^{n-1}(\lambda)$
such that $\C P^{n-1}(\lambda) \setminus B^{2n}(\lambda)$ is the hyperplane at
infinity. In this way, we obtain various improvements and generalizations
of the theorems concerning symplectic topology of polydisks proven in
\cite{hind-kerm}, \cite{hind}.

We give two high dimensional generalizations of \eqref{eq:fooo}.
Denote by $(z_1,\dots, z_n)$ the
complex coordinates of $\C^n \cong \R^{2n}$. We decompose
$$
(z_1,\dots, z_n) = (z_1, z')
$$
with $z' = (z_2, \dots, z_n)$. We denote
$$
D(a_1,a_2,\dots, a_n) = \{(z_1,\dots, z_n) \mid \pi |z_1|^2 < a_1, \dots,  \pi|z_n|^2 < a_n\} \subset \C^n
$$
where $a_1 \leq a_2 \leq \dots \leq a_n$. Hind \cite{hind} considers only the case when $n = 2$. We also denote
the cylinder over the disk
$|z_1|^2 \leq (a_1+\e)/\pi$ by
$$
Z_{1,n-1}(a_1+\e) = \{(z_1,\dots, z_n) \mid \pi |z_1|^2 < a_1 + \e\}
$$
for $\e > 0$ small. The following two theorems can be regarded as two
different high dimensional
generalizations of the lower bound in \eqref{eq:hind}.

\begin{thm}\label{thm:hindpoly} Let $S > 1$ and $0<\e < 1$.
Put $Z_{1,n-1} = Z_{1,n-1}(1 + \e)$. Then we have
$$
S \leq e^{Z_{1,n-1}}(D(1,S,\dots, S)).
$$
\end{thm}

\begin{thm}\label{thm:hindbi} Let $S > 1$ and $0<\e < 1$.
Let $k$ be an integer satisfying $1 \leq k < n$.
Put $Z_{n-k,k} = Z_{n-k,k}(1 + \e) = D^2(1+\e)^{n-k} \times \C^k$. Then we have
$$
S \leq e^{Z_{n-k,k}}(D^2(1)^{n-k} \times B^{2k}(kS)).
$$
\end{thm}
\par\smallskip

This paper borrows many notations and definitions from \cite{fooo:book}
without delving into detailed explanations thereof. Especially the notion of
bulk deformations is used mainly to make the statement of Theorem J
from \cite{fooo:book} in this paper as close as possible to that of \cite{fooo:book}.
We refer to the relevant sections of \cite{fooo:book} for more explantions
thereof.
However, for those who are mainly interested in the overall argument how the
torsion threshold can be used in the study of displacement energy, we
recommend them to directly look at Section \ref{sec:nondisplace},
Section \ref{sec:torthre} and Section \ref{sec:polydisks} and
refer to other sections as needed.

In March 2010, R. Hind gave a talk at MSRI workshop on Symplectic and Contact
Topology and Dynamics: Puzzles and Horizons.
Thanks to his talk, we took another look at our proof of Theorem 6.1.25 in \cite{fooo:book}
and found out an inaccurate point in the proof, which we rectify in this paper.
We thank him for his interesting talk and discussion.

\section{Notations}\label{sec:notation}

We introduce the
following general definitions and notations which we will use in this paper.

Let $(X, \omega)$ be a symplectic manifold. We denote by $J$
a time-dependent family of $\omega$-compatible almost
complex structures $J = \{J_t\}_{t \in [0,1]}$.

Let $\psi: X \to X$ be a Hamiltonian diffeomorphism and
$H \in C^\infty([0,1]\times X)$ a
\emph{normalized} Hamiltonian with $\phi_H^1 = \psi$ and $\int_X H_t
\, \omega^n = 0$. We denote the Hofer norm (see \cite{hofer}) of $H$ by
\be\label{def:HofernormH}
\|H\| =  \int_0^1 (\max H_t - \min H_t) \, dt.
\ee
We define the Hofer norm of $\psi$ by
\be\label{def:Hofernorm}
\|\psi\| = \inf_{H \mapsto \psi} \|H\|,
\ee
where $H \mapsto \psi$ means that
$\psi = \phi_H^1$.
We also define the length of a Hamiltonian isotopy $\phi_H=\{\phi^t_H\}$ by
\be\label{def:Hamleng}
{\rm leng}(\phi_H)=\|H\|.
\ee
Following Weinstein's notation \cite{alan}, we denote the set of Hamiltonian
deformations of $L$ by
$$
\frak{Iso}(L) =\{\psi(L) \mid \psi \in Ham(X,\omega)\}
$$
for a given Lagrangian submanifold $L \subset (X,\omega)$.

\begin{defn}\label{def:dist}
(1) Let $L' \in \frak{Iso}(L)$. Then we define the {\it Hofer distance} between $L, \, L'$ by
$$
\dist(L,L') = \inf_{H \in C^\infty([0,1] \times X)}
\{\|H\| \mid \phi_H^1(L) = L' \}.
$$
Or equivalently,
$$
\dist(L,L') = \inf_{\psi \in Ham(X,\omega)}
\{\|\psi\| \mid \psi(L) = L' \}.
$$
(2) Let $Y \subset X$.
We define the {\it displacement energy} $e^X(Y) \in [0,\infty]$ by
$$
e^X(Y) : = \inf_{\psi \in Ham(X,\omega)} \{ \Vert \psi\Vert \mid  \,
\psi(Y) \cap \overline Y = \emptyset\}.
$$
We put $e^X(Y) = \infty$ if there exists no $\psi \in Ham(X,\omega)$
with $\psi(Y) \cap \overline Y = \emptyset$.
When no confusion can occur, we simply write
$e(Y)$ instead of $e^X(Y)$.
\end{defn}

Let $\rho(\tau)$ be a smooth function on $\R$ such that $\rho(\tau)=0 \text{ or } 1$
when $\vert \tau \vert$ is sufficiently large.
In this paper we will take $\rho(\tau)$ as one of the following three types, either
\bea
\rho_+(\tau) & = & \begin{cases} 0 \quad & \mbox{for } \, \tau \leq 0 \\
1 \quad & \mbox{for } \, \tau \geq 1
\end{cases} \nonumber\\
\rho_+'(\tau) &\geq & 0 \label{eqn:rho}
\eea
 or $\rho_- = 1 -\rho_+$ or $\rho_K$ satisfying
\bea\label{eq:rho}
\rho_K(\tau) & = & \begin{cases} 0 \quad & \mbox{for } \, |\tau| \geq K  \\
1 \quad & \mbox{for }\, |\tau| \leq K -1
\end{cases} \nonumber\\
\rho_K' & \geq &  0 \quad \mbox{ on }\, [-K, -K+1], \quad \rho_K'
\leq 0 \quad\mbox{ on }\, [K-1,K]
\eea
for $K \geq 1$ and $\rho_K$
goes down to $\rho_{K=0} = 0$, e.g., $\rho_K = K\rho_1$. (See \eqref{rho_K}.)
\par
Ordering of the arguments in the pair $(L^{(0)},L^{(1)})$ varies
in the notations from \cite{fooo:book} for the various objects associated to
the pair of Lagrangian submanifold $(L^{(0)},L^{(1)})$. We mostly follow them
in the present paper. Specifically, we would like to mention the following
conventions:
\begin{enumerate}
\item (Path spaces) $\Omega (L^{(0)},L^{(1)};\ell_a)$.
\item (Floer moduli spaces) $\CM(L^{(1)},L^{(0)})$
\item (Floer complex and homology) $CF(L^{(1)},L^{(0)})$, $HF(L^{(1)},L^{(0)})$.
\end{enumerate}

\section{A non-displacement theorem}\label{sec:nondisplace}

In this section, we prove Theorem \ref{thm:non-hind}.
Recall that $S^1(S) \subset \C$ is a circle of radius $r$ with its area
$\pi r^2 = S$ and $S^1_{eq} \subset S^2(1)$ is the
equator in the sphere of area $1$. We put $X= \C \times S^2(1)$.
\begin{proof}[Proof of Theorem \ref{thm:non-hind}]
Suppose to the contrary that
\begin{equation}\label{empty}
\psi_H(S^1(S) \times S^1_{eq}) \cap (S^1(S) \times S^1_{eq}) =
\emptyset
\end{equation}
for a Hamiltonian $H$ with $\|H\| < S$. We denote
$$
c: = S - \|H\| > 0.
$$
Let $L^{(0)} = S^1(S) \times S^1_{eq}$. We choose $L^{(1)}$ as a small
Hamiltonian perturbation of $S^1(S) \times S^1_{eq}$ defined as follows:
We move $S^1(S)$ and $S^1_{eq}$ by small isometries on $\C$ and $S^2(1)$ respectively and obtain
$S^1_2(S), \,S^1_{2,eq}$ so that
$$
\psi_H(L^{(0)}) \cap (L^{(1)}) = \emptyset
$$
where we denote $L^{(1)} := S^1_2(S)\times S^1_{2,eq}$.

Since the condition \eqref{empty} is an open condition, such an
isometry obviously exists. We choose this perturbation so that
\be\label{dist}
\dist(L^{(0)},L^{(1)}) = \dist(S^1(S)
\times S^1_{eq},S^1_{2}(S)\times S^1_{2,eq}) \leq
\frac{c}{4}.
\ee

We then will deduce contradiction. We remark that $L^{(0)} \cap L^{(1)}$
consists of four transversal intersection points.
We take a one-parameter family of smooth functions $\rho_K$ on $\R$ satisfying (\ref{eq:rho}).
We denote by $X_H$ the Hamiltonian vector filed
of $H$ defined by $dH=\omega(X_H, \cdot)$.

Consider any pair $(p_-,p_+)$ of intersection points in
$ L^{(0)} \cap L^{(1)}$ and the equation
\begin{equation}\label{para}
\frac{\partial u}{\partial \tau} +
J\left(\frac{\partial u}{\partial t} - \rho_K (\tau) X_H(u)\right)
= 0
\end{equation}
of $u=u(\tau,t): \R \times [0,1] \to \C
\times S^2$ satisfying the boundary condition
\be\label{bdycond}
u(\tau,0)\in L^{(0)}, \, u(\tau,1) \in L^{(1)}, \,
u(\pm\infty,t) \equiv p_{\pm} \in L^{(0)} \cap L^{(1)}
\ee
and the finite energy condition
\begin{equation}\label{eqn:energy}
E_{(J,H,\rho_K)}(u) := \frac{1}{2} \int_{\R \times [0,1]}
\left|\frac{\del u}{\del \tau}\right|^2
+ \left|\frac{\partial u}{\partial t} - \rho_K (\tau) X_H(u)\right|^2 d\tau dt
 < \infty.
\end{equation}
Here we use the canonical complex structure $J$ on
$\C \times S^2(1)$, which we do not perturb.
Note that any solution $u$ of \eqref{para}, \eqref{bdycond}
carries a natural homotopy class. We denote
this homotopy class by $B$. As standard in Floer theory we define
the equivalence relation $B \sim B'$ if and only if
\be\label{equiv}
\omega(B_1) = \omega(B_2), \quad \mu(B_1) = \mu(B_2)
\ee
and denote by $\Pi(p_-,p_+)$ the set of equivalence classes.
Here $\mu$ denotes the Maslov index of the map $u$ associated to the pairs $(L^{(0)},L^{(1)})$
of Lagrangian submanifolds and the asymptotic condition $p_\pm$.
(We refer to Section 2.2 \cite{fooo:book} for a complete discussion on
the homotopy class and the Novikov covering.)
\par
The following energy estimate will be proved in Section
\ref{sec:improved}.
See Proposition \ref{prop:specialcase}.

\begin{lem} Let $0\leq K < \infty$ and $u$ be any finite energy solution of
\eqref{para}, \eqref{bdycond}. Then
\be \label{eq:energyHofer}
E_{(J,H,\rho_K)}(u) \leq \int u^*\omega + \|H\|.
\ee
\end{lem}

We consider the parameterized moduli space
$$
\mathcal M^{para}(p_{-},p_{+};B) = \bigcup_{K \in \R_{\geq 0}} \{K\} \times \mathcal M^K(p_{-},p_{+};B)
$$
where $\mathcal M^K(p_{-},p_{+};B)$ is the space of
solutions in class $B$ to the equation (\ref{para}) satisfying the conditions (\ref{bdycond}) and
(\ref{eqn:energy})
for the parameter $K \in [0,\infty)$.
We note that the symplectic area
$\int_{\R\times [0,1]} u^*\omega$ is invariant under the homotopy in the space of
\emph{smooth} maps with the boundary condition \eqref{bdycond}, which is \emph{fixed}.
We shall consider only the triples $(p_{-},p_{+};B)$ whose associated moduli space
$\mathcal M^{para}(p_{-},p_{+};B)$ has virtual dimension $0$ or $1$.
\par

To study $\mathcal M(p_-,p_+;B)=\mathcal M^{K=0}(p_-,p_+;B)$, we also need to study the equation
$$
\frac{\partial u}{\partial \tau} +
J \frac{\partial u}{\partial t} = 0
$$
with the same fixed Lagrangian boundary condition but possibly with different asymptotic
condition $(p_-',p_+')$. The associated moduli space
carries the natural $\R$-action and denote by
$$
\overline{\mathcal M}(p'_{-},p'_{+};B')
$$
the compactification of its quotient by this $\R$-action.
We shall however need to consider only those $(p'_{-},p'_{+};B')$ whose associated moduli
space has virtual dimension $0$.
\par
First since all the nontrivial holomorphic disk
which bounds either $L^{(0)}$ or $L^{(1)}$ have Maslov index $\ge 2$
and all the holomorphic spheres have Chern number  $\ge 2$ we can easily perturb
those moduli spaces (of virtual dimension 0 or 1) so that they
do not have disk or sphere bubble. We would like to point out that the
necessary transversality result on the moduli spaces can be easily achieved in the
current context: For $\mathcal M^{para}(p_{-},p_{+})$ we can perturb Hamiltonian term on compact set
to make it transversal. For $\overline{\mathcal M}(p_{-},p_{+})$
we can directly check that this moduli space is transversal using the fact that our
chosen almost complex structure $J$ is the standard integrable one on $\C \times S^2(1)$.
\par
Now we study the end and the boundary of $\mathcal M^{para}(p,p;0)$ for a $p \in L^{(0)} \cap L^{(1)}$.
The space $\mathcal M^{para}(p,p;0)$ is one dimensional. Here $0 \in \Pi(p,p)$ denotes the
equivalence class corresponding to the constant map $u\equiv p$. In particular, we have
\begin{equation}\label{0area}
\int u^*\omega = 0.
\end{equation}
\par
We observe that the condition
$\psi_H(L^{(0)}) \cap L^{(1)} = \emptyset$ implies the following
\begin{lem}
For all sufficiently large $K > 0$,  (\ref{para}) has no solution. Namely, we have $\mathcal M^K(p,p;0) =\emptyset$.
\end{lem}
\begin{proof}
We refer to the end of p. 901 of \cite{oh:mrl} for its proof.
\end{proof}

So the boundary of the compactified moduli space
$$
\overline{\mathcal M}^{para}(q,p;0)
$$
consist of the following three types:
\begin{enumerate}
\item $K=0$.
\item
\begin{equation}\label{leftslide}
\bigcup_{q} \overline{\mathcal M}(p,q;B_1) \times \mathcal M^{para}(q,p;B_2), \, B_1 + B_2 = 0.
\end{equation}
\item
\begin{equation}\label{rightslide}
\bigcup_{q}\mathcal M^{para}(p,q;B_1) \times \overline{\mathcal M}(q,p:B_2), \, B_1 + B_2 = 0.
\end{equation}
\end{enumerate}
The case $(1)$ gives rise to exactly one element. (That is the constant map, $u\equiv p$.)
Therefore the sum of the numbers of the boundaries of type
(\ref{leftslide}) and of (\ref{rightslide}) must be odd. We will show that
this is impossible.
\par
For this purpose, we examine each of the two types in detail.
We first consider the case of (\ref{leftslide}).
Let
$$
(v,(u,K_0)) \in \overline{\mathcal M}(p,q;B_1) \times \mathcal M^{para}(q,p;B_2).
$$
The energy bound
\eqref{eq:energyHofer} and \eqref{eqn:energy} yield the inequality
$$
0 \leq E_{(J,H;\rho_K)}(u) \leq \int u^*\omega + \|H\|
$$
for any solution of \eqref{para} for any $K$.
Therefore from $\Vert H\Vert = S - c < S$ and \eqref{dist}, we derive
\begin{equation}\label{ineq:energy}
\int u^*\omega \ge - \|H\|  = -S + c.
\end{equation}
\par
Since we consider an element $v$ of $\overline{\mathcal M}(p,q;B_1)$
whose virtual dimension is $0$, the element $v$ must be of the product form into
\begin{enumerate}
\item[{(a)}] $(v_1,\{pt\})$ with $v_1: \R \times [0,1] \to \C$ holomorphic
\item[{(b)}] $(\{pt\},v_2)$ with $v_2:\R \times [0,1] \to S^2(1)$ holomorphic.
\end{enumerate}

First consider the case (a).
We note that there are three bounded
connected components of $\C \setminus (S^1(S) \cup S^1_2(S))$.
Two of them say $D_1,\, D_2$ have small area and the third one $D_3$ has area $S - \epsilon$
for some $\epsilon > 0$. By choosing a small isometry that we uses
in the beginning to define $L_1$, we may choose $\epsilon \leq \frac{c}{3}$ so that
$$
\epsilon + \frac{c}{3} + \Vert H \Vert < S - \frac{c}{3}.
$$
\par
For this choice of $\e> 0$, we claim $D_3$ cannot appear in the compactification of $\mathcal
M^{para}(p,p;0)$. In fact, if it did, $u \in \mathcal M^{para}(q,p;B_2)$ and $D_3$ would give an
element of this compactification that lead to
$$
0 = \int u^* \omega + \int v^*\omega \geq -S + c + Area(D_3) \geq
-S + c + S - \frac{c}{3} = \frac{2c}{3} > 0,
$$
a contradiction.
\par
Note that $D_1$ and $D_2$ have the same area $\int_{D_i} \omega = \omega(B_1)$.
Therefore the end of $\mathcal M^{para}(q,p;0)$ will come in a pair
of the form $(v_\pm,(u,K_0))$ contained in
$$
\overline{\mathcal M}(p,q;B_1) \times \mathcal M^{para}(q,p;B_2)
$$
so that for each given
$(u,K_0) \in \mathcal M^{para}(q,p;B_2)$, there is a pair
$v_-,\, v_+ \in \overline{\mathcal M}(p,q;B_1)$.
Therefore the cardinality of this set must be even.
\par
For the case (b), similarly the end element again comes in a pair of the
form $((u,K_0),v_\pm)$ with the  the same area $\int v_-^*\omega = \int v_+^*\omega$.
(\emph{Here we use the fact that we put the equator on $S^2(1)$.})
Again the cardinality of this set is even.
\par
By the same argument, (\ref{rightslide}) consists of even number of points.
This is a contradiction.
\end{proof}

\section{Comparison of two Cauchy-Riemann equations and coordinate change}
\label{sec:comparison}

In this section, we explain what the \emph{coordinate change} we called in
the introduction means. To describe it precisely, we
briefly recall the Novikov covering spaces of the path spaces joining Lagrangian submanifolds,
on which the action functional will be defined.

We denote the path space by
$$
\Omega (L^{(0)},L^{(1)})= \{ \ell : [0,1] \to X ~\vert~ \ell(0) \in L^{(0)}, ~\ell(1) \in L^{(1)} \}.
$$
We first recall the action one-form $\alpha$ on $\Omega(L^{(0)},L^{(1)})$
defined by
\be\label{eq:actionform}
\alpha(\ell)(\xi) = \int_0^1 \omega(\dot \ell(t),\xi(t))\, dt.
\ee
This form is a `closed' one-form but not `exact' in general.
Due to the presence of the multi-valuedness of the associated action
functional we consider the Novikov covering space $\widetilde \Omega(L^{(0)},L^{(1)};\ell_a)$
of the connected component
$
\Omega(L^{(0)}, L^{(1)}; \ell_a)
$
containing a chosen base path $\ell_a \in \Omega(L^{(0)},L^{(1)})$ for each
$$
a \in \pi_0(\Omega(L^{(0)},L^{(1)})),
$$
and consider
its associated action functional
\be\label{eq:geom-action}
\CA_{\ell_a}([\ell,w]) = \int w^*\omega \quad \text{for} \quad [\ell, w] \in \Omega(L^{(0)},L^{(1)};\ell_a).
\ee
A simple computation shows that
\be\label{eq:dCA}
d\CA_{\ell_a} = - \pi^*\alpha
\ee
on $\widetilde \Omega(L^{(0)},L^{(1)};\ell_a)$.
Here $[\ell,w] \in \widetilde \Omega(L^{(0)}, L^{(1)}; \ell_a)$ is an
equivalence class of the pair $(\ell,w)$
of $\ell \in \Omega(L^{(0)},L^{(1)};\ell_a)$
and $w:[0,1]^2 \to X$ satisfying
$$
w(s,0) \in L^{(0)}, \, w(s,1) \in L^{(1)}, \, w(0,t) = \ell_a(t), \, w(1,t) = \ell(t)
$$
and  $\pi: \widetilde \Omega(L^{(0)},L^{(1)};\ell_a)
\to \Omega(L^{(0)},L^{(1)})$ is the natural projection given by $[\ell,w] \mapsto \ell$.
The equivalence relation is given as follows; $(\ell,w_1) \sim (\ell,w_2)$ if and only if
\be\label{eq:novikov}
\omega([\overline w_1 \# w_2]) = 0 = \mu_{L^{(0)}L^{(1)}}(\overline w_1 \# w_2)
\ee
where $\mu_{L^{(0)}L^{(1)}}$ is the Maslov index of the annulus map
$\overline w \# w': S^1 \times [0,1] \to X$ with boundary lying on $L^{(0)}$ at $t=0$ and
on $L^{(1)}$ at $t = 1$.
We refer to Definition 2.2.4 \cite{fooo:book} for the precise definition of
$\widetilde \Omega(L^{(0)}, L^{(1)}; \ell_a)$.

\begin{rem}\label{rem:connectedcomp}
In this section and the next, we pick and discuss one connected component of $\Omega(L^{(0)},L^{(1)})$
and its Novikov covering space without loss of generality.
In Section \ref{sec:correction}, we will consider
all connected components to study Floer chain complex
which is in fact a direct sum of Floer chain complex
for each connected component.
\end{rem}
\par\smallskip
Now
for a pair $(L^{(0)},L^{(1)})$ of compact
Lagrangian submanifolds we consider the Hamiltonian  deformation $(L^{(0)\prime},L^{(1)\prime})$ given by
$$
L^{(0)\prime} \in \frak{Iso}(L^{(0)}), \quad L^{(1)\prime} \in \frak{Iso}(L^{(1)}).
$$
We also consider a family $J^s = \{J^s_t\}_{0 \leq t \leq 1}$
of $\omega$-compatible almost complex structures.
We take Hamiltonian isotopies $\phi_{H^{(0)}} = \{\phi_{H^{(0)}}^s\}_{0 \leq s\leq 1},\,
\phi_{H^{(1)}} = \{\phi_{H^{(1)}}^s\}_{0 \leq s\leq 1}$ such that
\be\label{eq:psi01}
\phi_{H^{(0)}}^1(L^{(0)}) = L^{(0)\prime}, \quad \phi_{H^{(1)}}^1(L^{(1)}) = L^{(1)\prime}.
\ee
For a given pair of Hamiltonian isotopies
$\phi_{H^{(i)}}$ for $i = 0, \, 1$, $J^s =
\{J^s_t\}$, and a given smooth function $\rho$ as in
Section \ref{sec:notation}, we consider moving Lagrangian boundary
value problem \be\label{eq:moving}
\begin{cases}\dudtau + J^\rho \dudt = 0 \\
u(\tau,0) \in \phi_{H^{(0)}}^{\rho(\tau)}(L^{(0)}), \, u(\tau,1) \in
\phi_{H^{(1)}}^{\rho(\tau)}(L^{(1)})
\end{cases}
\ee where $J^\rho(\tau,t) = J^{\rho(\tau)}_t$.
Let $[\ell_p,w] \in \widetilde \Omega(L^{(0)},L^{(1)};\ell_a)$ and $[\ell_{q'},w']
\in
\widetilde \Omega(L^{(0)\prime},L^{(1)\prime};\ell_a^{\prime})$.
Here we choose
\be
\ell_a' (t) :=\phi_{H^{(1)}}^1(\phi^{1-t}_{H^{(1)}})^{-1} \circ \phi_{H^{(0)}}^1(\phi^t_{H^{(0)}})^{-1} (\ell_a(t))
\ee
as the base path of
$\Omega (L^{(0)\prime},L^{(1)\prime})$.
See \eqref{eq:gauge1} and \eqref{eq:gauge2} below.
(We used the notation $\ell_0^{(\psi^{(0)},\psi^{(1)})}$
in \cite{fooo:book}.)
We denote by
\be\label{eq:MMmoving}
\CM^\rho((L^{(1)},\phi^1_{H^{(1)}}),(L^{(0)},\phi^1_{H^{(0)}});[\ell_p,w],[\ell_{q'},w'])
\ee
the set of solutions of \eqref{eq:moving} with
$$
[\ell_{q'}, I_{\phi^1_{H^{(0)}},\phi^1_{H^{(1)}}}^{\rho} (w\# u)] = [\ell_{q'},w'],
$$
where
$w \# u:[0,1] \times [0,1] \to X$ is the concatenation in the $\tau$-direction of
$w$ and $u$ and
$$
I_{\phi^1_{H^{(0)}},\phi^1_{H^{(1)}}}^{\rho}v (\tau,t) = \left(\phi_{H^{(1)}}^1(\phi_{H^{(1)}}^{\rho(\tau)(1-t)})^{-1} \circ \phi_{H^{(0)}}^1(\phi_{H^{(0)}}^{\rho(\tau)t})^{-1}\right) v (\tau, t).
$$ (see p. 308 in \cite{fooo:book}).
We note that we do not use the equation
\eqref{eq:moving} of {\it moving} Lagrangian boundary
value problem
and the moduli space
\eqref{eq:MMmoving} themselves in this article,  while in
Subsection 5.3.2 of \cite{fooo:book} we used them for construction of a filtered
$A_{\infty}$ bimodule homomorphism.
\par
The main goal of Section \ref{sec:comparison} - Section \ref{sec:correction} is to construct a suitable filtered $A_{\infty}$ bimodule
homomorphism, hence a chain map by considering the Hamiltonian perturbed Cauchy-Riemann equation
\eqref{eq:CRJH} different from \eqref{eq:moving},
when the pair $(L^{(0)},L^{(1)})$ is unobstructed in the sense of
Lagrangian Floer cohomology theory,
$$
C(L^{(1)},L^{(0)};\Lambda_{\text{\rm nov}}) \to
C(L^{(1)\prime},L^{(0)\prime};\Lambda_{\text{\rm nov}})
$$
in the point of view of filtration changes.
See the beginning of Section \ref{sec:correction}
for a quick review of
$C(L^{(1)},L^{(0)};\Lambda_{\text{\rm nov}})$.
The reason why we use the equation \eqref{eq:CRJH}
instead of \eqref{eq:moving} to construct a desired
filtered $A_{\infty}$ bimodule homomorphism
is to apply the improved estimates for solutions
of \eqref{eq:CRJH} carried out in Section \ref{sec:improved}.

First, we consider the particular pairs
$$
(L^{(0)},L^{(1)}), \quad (L^{(0)\prime},L^{(1)\prime}) = (\phi_H^1(L^{(0)}), L^{(1)})
$$
to explain the meaning of the {\it coordinate change}.
These particular pairs correspond to the special case $H^{(1)} \equiv 0$
in the general  discussion above.

We would like to compare the \emph{geometric} version
of Floer theory and the \emph{dynamical} one. Such a comparison is by now well-known, and
the case
\be\label{eq:LLphiH1L}
(L^{(0)},L^{(1)}), \quad (\phi_H^1(L^{(0)}), L^{(1)})
\ee
was exploited previously in \cite{oh:jdg} for the exact Lagrangian submanifolds
on the cotangent bundle in relation to the study of spectral invariants.
Here we need such a study for general compact Lagrangian submanifolds on
general $(X,\omega)$.
\par
The geometric version of the Floer complex for
$(L^{(0)\prime}=\phi_H^1( L^{(0)}), L^{(1)\prime}=
L^{(1)})$ is generated by
the intersection points
$$
\phi_H^1( L^{(0)}) \cap L^{(1)}
$$
and its Floer boundary map is constructed by the moduli space of genuine Cauchy-Riemann
equation
\be\label{eq:CR-geom}
\begin{cases}
\frac{\del u'}{\del \tau} + J' \frac{\del u'}{\del t} = 0\\
u'(\tau,0) \in \phi_H^1(L^{(0)}),\quad u'(\tau,1) \in L^{(1)}.
\end{cases}
\ee
Here $J'=J'_t = (\phi_H^1 (\phi_H^t)^{-1})_{\ast}J_t$.
We denote by $\MM(L^{(1)}, \phi_H^1(L^{(0)});J')$ the moduli space of finite
energy solutions of this equation.
Due to the presence of the multi-valuedness of the associated action
functional, we need to consider these equations
on the Novikov covering space of some specified connected component
$$
\Omega(\phi_H^1(L^{(0)}), L^{(1)}; \ell_a')
$$
with the base path $\ell_a' \in \Omega(\phi_H^1(L^{(0)}), L^{(1)})$, which is given by
\be\label{eq:ellprime}
\ell_a'(t)=\phi_H^1(\phi_H^t)^{-1}(\ell_a(t)),
\ee
and consider the
action functional
\be\label{eq:geom-action}
 \CA_{\ell_a'}([\ell',w'])= \int (w')^*\omega
\ee
where $[\ell',w'] \in \widetilde \Omega(\phi_H^1(L^{(0)}), L^{(1)}; \ell_a')$ and
$w':[0,1]^2 \to X$ is a map satisfying the boundary condition
$$
w'(0,t) = \ell_a'(t),\, w'(1,t) = \ell'(t), \, w'(s,0) \in \phi_H^1(L^{(0)}), \, w'(s,1) \in L^{(1)}.
$$
On the other hand the dynamical version of the Floer complex is generated by
the solutions of Hamilton's equation
\be\label{eq:hameq}
\dot x = X_H(t,x), \quad x(0) \in L^{(0)}, \, x(1) \in L^{(1)}
\ee
and its boundary map is constructed by the moduli space of perturbed Cauchy-Riemann
equation
\be\label{eq:CRJH}
\begin{cases}
\frac{\del u}{\del \tau} + J \left(\frac{\del u}{\del t} - X_H(u) \right) = 0\\
u(\tau,0) \in L^{(0)}, \quad u(\tau,1) \in L^{(1)}.
\end{cases}
\ee
We denote by $\MM(L^{(1)},L^{(0)};H;J)$ the moduli space of finite
energy solutions of this equation.
The action functional $\CA_{H,\ell_a}$ is defined by
\be\label{eq:dynam-action}
\CA_{H,\ell_a}([\ell,w])= \int w^*\omega + \int_0^1 H(t,\ell(t))\, dt
\ee
on $\widetilde\Omega(L^{(0)},L^{(1)};\ell_a)$.

These two Floer theories are related by the following
transformations of the bijective map
$$
\frak g_{H;0}^+:\widetilde\Omega(\phi_H^1(L^{(0)}), L^{(1)}; \ell_a') \to
\widetilde\Omega(L^{(0)},L^{(1)};\ell_a); \quad [\ell',w'] \mapsto [\ell,w]
$$
given by the assignment
\be\label{eq:gauge1}
\ell(t) = \phi_H^t (\phi_H^1)^{-1}(\ell'(t)), \quad
w(s,t)=\phi_H^t (\phi_H^1)^{-1}(w'(s,t)).
\ee
This provides a bijective correspondence of the critical points
\be\label{eq:Crit-bijective}
\Crit \CA_{\ell_a'} \longleftrightarrow \Crit \CA_{H,\ell_a}; \quad [p',w'] \mapsto
[z^{H}_{p'},w]
\ee
with $p' \in \phi_H^1(L^{(0)}) \cap L^{(1)}$, $z^H_{p'}(t):= \phi_H^t (\phi_H^1)^{-1}(p')$ and $w=\phi_H^t (\phi_H^1)^{-1}(w'(s,t))$,
and of the moduli spaces
$$
\MM(L^{(1)},\phi_H^1(L^{(0)});J') \mapsto \MM(L^{(1)},L^{(0)};H;J)
$$
with $J_t = (\phi_H^t(\phi_H^1)^{-1})_*J'_t$ where the map
is defined by
$$
u(\tau,t)= \phi_H^t (\phi_H^1)^{-1}(u'(\tau,t)).
$$
The map $\frak g_{H;0}^+$ also preserves the action up to a constant in that
\begin{lem} \label{lem:gH0+}
Denote
$$
c(H;\ell_a) := \int_0^1 H(t,\ell_a(t))\, dt
$$
which is a constant depending only on $H$ and the base path $\ell_a$ of the connected component
$\Omega(L^{(0)},L^{(1)};\ell_a)$.
Then
\be\label{eq:actiondiff}
\CA_{H,\ell_a}\circ \frak g_{H;0}^+([\ell',w']) = \CA_{\ell_a'}([\ell',w']) +
c(H;\ell_a)
\ee
on $\widetilde{\Omega}(L^{(0)},L^{(1)};\ell_a)$.
\end{lem}

\begin{rem}\label{rem:constant}
Since we normalized Hamiltonians so that
$\int_X H_t \omega^n =0$ for each $t$ in Section
\ref{sec:notation},
we can take $\ell_a$ for each connected component
of $\Omega(L^{(0)},L^{(1)})$ in such a way
that $c(H;\ell_a)=\int_0^1 H(t,\ell_a (t) )dt = 0$.
It is not essential to choose $\ell_a$ in a way as above.
In fact, if we take a based path so that $c(H;\ell_a) \ne 0$,
it is enough to include an extra term $c(H;\ell_a)$ in the energy estimate
on $\Omega(L^{(0)},L^{(1)};\ell_a)$
when we apply the coordinate change $\frak g_{H;0}^{+}$ or its inverse.
Since we will consider all connected components of $\Omega(L^{(0)},L^{(1)})$
in Section \ref{sec:correction} (see Remark \ref{rem:connectedcomp}),
we have to add the constant
$c(H;\ell_a)$  for  each connected component  $\Omega(L^{(0)},L^{(1)};\ell_a)$.
Thus, to avoid heavy notation, we simply choose $\ell_a$ so that  $c(H;\ell_a) =0$ for each
connected component  $\Omega(L^{(0)},L^{(1)};\ell_a)$.
\end{rem}

\begin{proof} The proof is by a direct calculation.
Let $[\ell',w'] \in \widetilde\Omega(\phi_H^1(L^{(0)}), L^{(1)}; \ell_a')$.
Then
$$
\CA_{H;\ell_a}(\frak g_{H;0}^+([\ell',w']))  =
\int w^*\omega + \int_0^1 H(t,\ell(t))\, dt
$$
and
$$
w^*\omega = \omega \left(\frac{\del w}{\del s},\frac{\del w}{\del t}\right) ds \wedge dt.
$$
We compute
\beastar
\frac{\del w}{\del s} &= &d(\phi_H^t(\phi_H^1)^{-1})\left(\frac{\del w'}{\del s}\right)\\
\frac{\del w}{\del t} &= &d(\phi_H^t(\phi_H^1)^{-1})\left(\frac{\del w'}{\del t}\right) + X_{H_t}(w(s,t)).
\eeastar
Substituting this into the above, we obtain
\beastar
\int w^*\omega & = &\int_0^1\int_0^1 \omega\left(\frac{\del w}{\del s},\frac{\del w}{\del t}\right)\, ds\,dt\\
&=& \int_0^1\int_0^1 \omega\left(\frac{\del w'}{\del s},\frac{\del w'}{\del t}\right)\, ds\,dt \\
& {} & \qquad + \int_0^1\int_0^1\omega\left(d\left(\phi_H^t(\phi_H^1)^{-1}\right)\left(\frac{\del w'}{\del s}\right),
X_{H_t}(w(s,t))\right)\, ds\, dt \\
& = & \int (w')^*\omega - \int_0^1\int_0^1 dH_t(w(s,t))\left(d(\phi_H^t(\phi_H^1)^{-1})
\frac{\del w'}{\del s}\right) \, ds\,dt\\
& = & \int (w')^*\omega - \int_0^1 \int_0^1 \frac{\del}{\del s} H_t(w(s,t))\, ds\, dt \\
& = & \int (w')^*\omega - \int_0^1 H_t(w(1,t)) \, dt + \int_0^1 H_t(w(0,t))\, dt.
\eeastar
Substituting this into the above definition of $\CA_{H;\ell_a}(\frak g_{H;0}^+([\ell',w']))$,
the proof is finished.
\end{proof}

We denote by $\frak g_{H;0}^-$ the inverse $\frak g_{H;0}^- = (\frak g_{H;0}^+)^{-1}$.
The outcome of the above discussion is that the two associated Floer cohomologies
are isomorphic to each other.

\par
So far we have moved the first argument $L^{(0)}$ in the pair $(L^{(0)},L^{(1)})$.
We can also move the second argument $L^{(1)}$ instead. In that case, we define the coordinate
change transformation by
$$
\frak g_{\widetilde H;1}^+:\widetilde\Omega(L^{(0)}, \phi_H^1(L^{(1)}); \ell_a') \to
\widetilde\Omega(L^{(0)},L^{(1)};\ell_a); \quad [\ell',w'] \mapsto [\ell,w]
$$
given by the assignment
\be\label{eq:gauge2}
\ell(t) = \phi_H^{1-t}(\phi_H^1)^{-1}(\ell'(t)), \quad
w(s,t)=\phi_H^{1-t} (\phi_H^1)^{-1}(w'(s,t))
\ee
where $\widetilde H(t,x) : = - H(1-t,x)$ is the Hamiltonian
generating the latter Hamiltonian path $ t \mapsto \phi_H^{1-t} (\phi_H^1)^{-1}$.
This provides a bijective correspondence
$$
\Crit \CA_{\ell_a'} \longleftrightarrow \Crit \CA_{\widetilde H,\ell_a}
$$
and the moduli spaces
$$
\MM(\phi_H^1(L^{(1)}),L^{(0)};J') \mapsto \MM(\widetilde H;L^{(1)},L^{(0)};J)
$$
with $\widetilde J_t = (\phi_H^{1-t}(\phi_H^1)^{-1})_*J'_t$.
Here $\MM(\widetilde H;L^{(1)},L^{(0)};\widetilde J)$ is the moduli space of solutions of
\be\label{eq:CRtildeHJ}
\begin{cases}
\dudtau + \widetilde J\left(\dudt - X_{\widetilde H}(u)\right) = 0 \\
u(\tau,0) \in L^{(0)}, \, u(\tau,1) \in L^{(1)}.
\end{cases}
\ee
The action functional $\CA_{\widetilde H,\ell_a}$ is given by
\be\label{eq:CAtildeH}
\CA_{\widetilde H, \ell_a}([\ell,w]) =
\int (\widetilde w)^*\omega + \int_0^1 \widetilde H(t,\ell(t))\,d t.
\ee
The explicit formula of the
latter correspondence is given by
$$
u(\tau,t)) = \phi_H^{1-t} (\phi_H^1)^{-1}(u'(\tau,t)).
$$
Again the following can be proved by a similar computations used to prove Lemma \ref{lem:gH0+}
whose proof is left to the readers.

\begin{lem}\label{lem:gH1+} We have
\be\label{eq:tildeactiondiff}
\CA_{\widetilde H,\ell_a}\circ \frak g_{\widetilde H;1}^+([\ell',w']) = \CA_{\ell_a'}([\ell',w']) +
c(\widetilde H;\ell_a)
\ee
where
$$
c(\widetilde H;\ell_a) := \int_0^1 \widetilde H (t,\ell_a(t))\, dt
$$
is a constant depending only on $H$ and the base path $\ell_a$.
\end{lem}
We denote by $\frak g_{\widetilde H;1}^-$ the inverse of $\frak g_{\widetilde H;1}^+$.

\section{Improved energy estimate}
\label{sec:improved}

First we consider the case of varying the first argument $L^{(0)}$
of the pair $(L^{(0)}, L^{(1)})$
$
(L^{(0)\prime},L^{(1)\prime}) = (\phi_H^1(L^{(0)}),L^{(1)}).
$
In this case, as far as the study of the optimal filtration change is concerned,
employing the moduli space with moving Lagrangian boundary is not the best one.
We will show that employing the standard \emph{perturbed} Cauchy-Riemann equation
by Hamiltonian vector fields with \emph{fixed} Lagrangian boundaries,
which intertwines the geometric version and the dynamical version of the
Floer complex, gives a stronger energy estimate which
gives rise to the optimal change of filtration.

Let $\rho$ be one of smooth functions on $\R$ of the three types
introduced in Section \ref{sec:notation}.
See \eqref{eqn:rho} and \eqref{eq:rho}.
Consider the perturbed Cauchy-Riemann equation
\begin{equation}\label{eq:CRJH}
\begin{cases}
\frac{\partial u}{\partial \tau} + J
\left(\frac{\partial u}{\partial t} -  \rho (\tau) X_H(u)\right) =0\\
u(\tau,0) \in L^{(0)},\, u(\tau,1) \in L^{(1)}
\end{cases}
\end{equation}
with the finite energy
$
E_{(J,H,\rho)}(u) < \infty.
$
The following a priori energy bound is a key ingredient in relation to
the lower bound of displacement energy. This kind of optimal estimate
is originally due to Chekanov \cite{chekanov}, which is
the key calculation that relates the energy and the Hofer norm
in an optimal way.  For completeness' sake, we include its proof
which is a slight variation of the calculation carried out in p. 901 \cite{oh:mrl}.
It is useful to decompose $\|H\|$ into two parts
$$
E^-(H) = \int_0^1 -\min \; H_t\, dt, \quad
E^+(H) = \int_0^1 \max \; H_t\, dt,
$$
which are so called the negative and positive part of the Hofer norm $\|H\|$.

\begin{lem}\label{lem:energyestimate}
Let $\rho = \rho_+$  as in \eqref{eqn:rho}.
Let $u$ be any finite energy solution of \eqref{eq:CRJH}. Then we have
\bea\label{eq:EJrho}
E_{(J,H,\rho)} (u) & = & \int u^*\omega + \int_0^1 H(t, u(\infty,t))\, dt \nonumber\\
&{}& \quad - \int_{-\infty}^{\infty} \rho'(\tau) \int_0^1 (H_t \circ u)\,dt\,d\tau.
\eea
\end{lem}
\begin{proof} The proof will be carried out by an explicit calculation.
We compute
\beastar
&{}& E_{(J,H,\rho)} (u) = \int_{-\infty}^\infty \int_0^1 \left|\frac{\partial u}{\partial
\tau}\right|^2_J \, dt\, d\tau
= \int_{-\infty}^\infty \int_0^1
\omega\left( \frac{\partial u}{\partial \tau},J\frac{\partial
u}{\partial \tau}\right)\,dt\, d\tau \\
&=& \int_{-\infty}^\infty \int_0^1 \omega\left(\frac{\partial
u}{\partial \tau}, \frac{\partial u}{\partial t}
-  \rho(\tau)X_{H_t}(u)\right)
\,dt\, d\tau \\
&=& \int_{-\infty}^\infty \int_0^1 \omega\left( \frac{\partial
u}{\partial \tau}, \frac{\partial u}{\partial t}\right)\, dt\,d\tau
-
\int_{-\infty}^\infty  \rho(\tau) \int_0^1 \omega\left(
\frac{\partial u}{\partial \tau},X_{H_t(u)}\right) \, dt\, d\tau \\
&=& \int u^*\omega - \int_{-\infty}^\infty  \rho(\tau) \int_0^1
\left(-dH_t(u) \frac{\partial u}{\partial \tau}\right)\,dt\,d\tau \\
&=&  \int u^*\omega + \int_{-\infty}^\infty \rho(\tau) \int_0^1
\frac{\partial}{\partial \tau}(H_t \circ u)\,dt\,d\tau \\
&=&  \int u^*\omega + \int_0^1 H(t, u(\infty,t))\, dt
- \int_{-\infty}^\infty
\rho'(\tau) \int_0^1 (H_t \circ u)\,dt\,d\tau.
\eeastar
Here at the last equality, we do integration by parts
and use the fact $\rho(\infty) = 1, \, \rho(-\infty) = 0$.
This finishes the proof of \eqref{eq:EJrho}.
\end{proof}

This lemma gives rise to the following key formula of the action difference
$$
\CA_{H,\ell_a}(u(\infty)) - \CA_{\ell_a}(u(-\infty)).
$$
\begin{prop}\label{prop:energyid} Let $p \in L^{(0)} \cap L^{(1)}$ and $q' \in \phi_H^1(L^{(0)}) \cap L^{(1)}$.
Denote by $z_{q'}^H \in \Omega(L^{(0)},L^{(1)};\ell_a)$ the Hamiltonian trajectory defined by
$$
z_{q'}^H(t) = \phi_H^t(\phi_H^{-1}(q'))
$$
and consider $[\ell_p,w] \in \Crit \CA_{\ell_a}$,
$[z_{q'}^H,w'] \in \Crit \CA_{H,\ell_a}$.
Suppose that $u$ is any finite energy solution of (\ref{eq:CRJH}) with $\rho=\rho_+$ as in \eqref{eqn:rho}
satisfying the asymptotic condition
and homotopy condition
\bea\label{eq:asympt}
u(-\infty) = \ell_p, \quad u(\infty) = z^H_{q'}, \quad w \# u \sim w'.
\eea
Then we have
\be\label{eq:energyid}
\CA_{H,\ell_a}([z_{q'}^H,w']) - \CA_{\ell_a}([w,\ell_p])
= E_{(J,H,\rho)}(u) + \int_{-\infty}^\infty \int_0^1 \rho'(\tau) H(t, u(\tau,t))\, dt\, d\tau.
\ee
\end{prop}
\begin{proof} By \eqref{eq:asympt}, we obtain
$$
\int u^*\omega = \int (w')^*\omega - \int w^*\omega.
$$
By substituting this into \eqref{eq:EJrho} and rearranging the resulting
formula, we obtain \eqref{eq:energyid} from the definitions \eqref{eq:geom-action} of $\CA_{\ell_a}$
and \eqref{eq:dynam-action} of $\CA_{H,\ell_a}$.
\end{proof}

An immediate corollary is
\begin{cor}\label{cor:actionchange} Suppose that there is a
solution $u$ of \eqref{eq:CRJH}
for $\rho=\rho_+$ as in Proposition \ref{prop:energyid}.
Then we have
\be\label{eq:actionrho}
\CA_{H,\ell_a}([z^H_{q'}, w']) - \CA_{\ell_a} ([\ell_p,w])  \geq \int_0^1 \min H_t\, dt = - E^-(H).
\ee
Similarly if there is a solution $u$ of \eqref{eq:CRJH} for $\rho = \rho_- = 1-\rho_+$,
\be\label{eq:actiontilderho}
\CA_{\ell_a}([\ell_{q},w]) - \CA_{H,\ell_a}([z^H_{p'},w']) \geq \int_0^1- \max H_t\, dt = - E^+(H).
\ee
\end{cor}

Next let us concatenate  the two equation \eqref{eq:CRJH} for $\rho_+$ as in \eqref{eq:rho} and $\rho_-=1-\rho_+$
by considering one-parameter family of elongation function of the type
\bea \label{rho_K}
\rho_K = \begin{cases}
\rho_+(\cdot +K)  & \mbox{for }\, \tau \leq -K+1\\
\rho_-(\cdot -K+1) & \mbox{for }\, \tau \geq K-1\\
1 & \mbox{for } \vert \tau \vert \leq K -1
\end{cases}
\eea
for $1 \leq K \leq \infty$ and further deforming $\rho_{K=1}$ further down to $\rho_{K=0}\equiv 0$.

\begin{prop}\label{prop:specialcase}
Let $u$ be a finite energy solution for \eqref{eq:CRJH} of the elongation
$\rho_K$ with asymptotic condition
$$
u(-\infty) = \ell_p, \, u(\infty) = \ell_q, \, w_-\# u \sim w_+.
$$
Then we have
\bea
\CA_{\ell_a}([\ell_q,w_+]) - \CA_{\ell_a}([\ell_p,w_-]) & \geq & -\left(E^-(H) + E^+(H)\right) = - \|H\| \label{eq:actionHofer}\\
E_{(J,H;\rho_K)}(u) & \leq & \int u^*\omega + \|H\|.
\eea
\end{prop}

\medskip

So far in this section, we have moved the first argument $L^{(0)}$ in the pair $(L^{(0)},L^{(1)})$.
When we move the second argument $L^{(1)}$ instead,
the only difference occurring in the above discussion will be the interchange
$$
-\min H \longleftrightarrow \max H.
$$
Now we move $L^{(0)}$ and $L^{(1)}$ by Hamiltonian isotopies
$\phi_{H^{(0)}}^t$ and $\phi_{H^{(1)}}^t$, respectively.
$$
(L^{(0)},L^{(1)}) \mapsto (L^{(0) \prime}=\phi^1_{H^{(0)}}(L^{(0)}), L^{(1) \prime}=\phi^1_{H^{(1)}}(L^{(1)})).
$$
Then we have the following bijection

$$
{\mathfrak g}_{H^{(0)},H^{(1)}}^+: (\ell'',w'') \in \widetilde{\Omega}(L^{(0) \prime},L^{(1) \prime})
\mapsto (\ell, w) \in  \widetilde{\Omega}(L^{(0)},L^{(1)}),
$$
where
$$
\ell(t)= \phi_{H^{(1)}}^{1-t}\circ (\phi_{H^{(1)}}^1)^{-1} \circ
\phi_{H^{(0)}}^t \circ (\phi_{H^{(0)}}^1)^{-1}(\ell''(t))
$$
and
$$
w(s,t)=\phi_{H^{(1)}}^{1-t}\circ (\phi_{H^{(1)}}^1)^{-1} \circ
\phi_{H^{(0)}}^t \circ (\phi_{H^{(0)}}^1)^{-1}(w''(s,t)).
$$
We write ${\mathfrak g}_{H^{(0)},H^{(1)}}^-=({\mathfrak g}_{H^{(0)},H^{(1)}}^+)^{-1}$.
By an abuse of notation, we also denote by ${\mathfrak g}_{H^{(0)},H^{(1)}}^{\pm}$
the bijection between the path spaces
${\Omega}(L^{(0)},L^{(1)})$ and ${\Omega}(L^{(0) \prime},L^{(1) \prime})$.
Then we obtain the following improved energy estimate.
Here we take the base path $\ell_a$ in such a way that
\be\label{eq:cHhat}
c(\widehat{H};\ell_a)=\int_0^1 \widehat H(t,\ell_a(t)) dt = 0
\ee
as in Remark \ref{rem:constant}.
Here
$\widehat{H}$ is the normalized Hamiltonian
generating
$$
 \phi_{H^{(1)}}^{1-t} \circ  (\phi_{H^{(1)}}^1)^{-1}  \circ \phi_{H^{(0)}}^t.
$$
The Hamiltonian $\widehat{H}$ is explitly written as
\be \label{hatH}
\widehat{H}(t,x)=-H^{(1)}(1-t,x)
+H^{(0)}(t,(\phi_{H^{(1)}}^{1-t} \circ  (\phi_{H^{(1)}}^1)^{-1})^{-1}(x)).
\ee
For $v:[0,1] \times [0,1] \to X$, we put
$$(\Phi_{H^{(0)},H^{(1)}} v) (s,t)=  \phi_{H^{(0)}}^1 \circ (\phi_{H^{(0)}}^t)^{-1}
\circ \phi_{H^{(1)}}^1 \circ (\phi_{H^{(1)}}^{1-t})^{-1}v(s,t).$$
By the expression \eqref{hatH} of $\widehat{H}$, we find that
\be
E^-(\widehat{H}) \leq E^-(H^{(0)})+E^+(H^{(1)}), \quad
E^+(\widehat{H}) \leq E^+(H^{(0)})+E^-(H^{(1)}).
\ee
Recall that we have chosen $\ell_a$ such that \eqref{eq:cHhat}
is satisfied and put
$\ell_a'' = {\mathfrak g}_{H^{(0)},H^{(1)}}^-(\ell_a).$

\begin{prop}\label{prop:improvedestimate} Let $(L^{(0)},L^{(1)})$ be a pair of compact
Lagrangian submanifolds and $(L^{(0)\prime},L^{(1)\prime})$
another pair with
$$
L^{(0)\prime} = \phi_{H^{(0)}}^1(L^{(0)}), \quad L^{(1)\prime} =  \phi_{H^{(1)}}^1(L^{(1)})
$$
and let $H^{(0)}$, $H^{(1)}$ be the normalized Hamiltonians generating
$\phi^1_{H^{(0)}}$ and $\phi^1_{H^{(1)}}$ respectively.
Consider a pair $[\ell_p,w] \in \widetilde
\Omega(L^{(0)},L^{(1)};\ell_a)$ and $[\ell_{q''},w'']
\in \widetilde\Omega(L^{(0)\prime},L^{(1)\prime};\ell_a^{\prime\prime})$ for which there exists a solution $u$ of
\eqref{eq:CRJH} with $\rho=\rho_+$ as in \eqref{eqn:rho} such that
$$
\lim_{\tau \to -\infty} u(\tau, \cdot )=\ell_p,  \  \lim_{\tau \to +\infty} u(\tau,\cdot )=
{\mathfrak g}_{H^{(0)},H^{(1)}}^{+} (\ell_{q''}), \
\Phi_{H^{(0)},H^{(1)}} (w \#u) \sim w''.
$$
Then we have
\be\label{eq:improvedestimate}
\mathcal A_{\ell_a''}([\ell_{q''},w'']) -
\mathcal A_{\ell_a}([\ell_p,w]) \geq -(E^-(H^{(0)}) + E^+(H^{(1)})).
\ee
Similarly, let
$[\ell_{p''},w''] \in \widetilde\Omega(L^{(0)\prime},L^{(1)\prime};\ell_a^{\prime\prime})$
and $[\ell_q,w] \in \widetilde \Omega(L^{(0)},L^{(1)};\ell_a)$.
If there exists a solution $u$ of \eqref{eq:CRJH} with $\rho=\rho_-=1- \rho_+$
such that
$$
\lim_{\tau \to -\infty} u(\tau, \cdot )=
{\mathfrak g}_{H^{(0)},H^{(1)}}^{+} (\ell_{p''}),
\  \lim_{\tau \to +\infty} u(\tau,\cdot )=\ell_{q}, \
\Phi_{H^{(0)},H^{(1)}} ^{-1}(w'') \#u \sim w,
$$
we have
\be\label{eq:improvedestimate'}
\mathcal A_{\ell_a}([\ell_{q}, w]) - \mathcal A_{\ell_a''}([\ell_{p''},w'']) \geq -(E^+(H^{(0)}) + E^- (H^{(1)})).
\ee
\end{prop}

The following proposition is parallel to Proposition \ref {prop:specialcase}

\begin{prop}\label{prop:finalestimate}
Let $(L^{(0)},L^{(1)})$ be a pair of compact Lagrangian submanifolds.  If there exists a solution
$u$ of \eqref{eq:CRJH} with $\rho=\rho_K$ in \eqref{rho_K} satisfying
$$
\lim_{\tau \to -\infty}u(\tau,\cdot)=\ell_p, \ \lim_{\tau \to +\infty}u(\tau,\cdot)=\ell_q, \
w_- \# u \sim w_+.
$$
Then we have
\be \label{action bound}
\CA_{\ell_a}([\ell_q,w_+]) - \CA_{\ell_a}([\ell_p,w_-])  \geq -\left( \| H^{(0)}\| + \| H^{(1)} \|\right)
\ee
and
\be \label{energy bound}
E_{(J,\widehat{H},\rho_K)}(u) \leq
\int u^{\ast} \omega +
\|H^{(0)}\| + \| H^{(1)}\|.
\ee
\end{prop}
\section{Corrected proofs of Theorem J
and Theorem 6.1.25 \cite{fooo:book}}
\label{sec:correction}

To keep the statement of Theorem J \cite{fooo:book} as it is, we need to modify
construction of the chain map used in the proof of Theorem 6.1.25 \cite{fooo:book}.

In the rest of the paper, we assume that
a Lagrangian submanifold is closed and relatively spin and a pair of Lagrangian submanifolds
is relatively spin (Definition 1.6 \cite{fooo:book})
unless otherwise noted. In this discussion we use the $\C$-coefficients
as in \cite{fooo:toric1,fooo:toric2} but
one can also use the $\Q$-coefficients as in \cite{fooo:book}.

We first recall the definition of the universal Novikov ring $\Lambda_{\text{\rm nov}}$ used in \cite{fooo:book}.
An element of $\Lambda_{\text{\rm nov}}$ is a formal sum
$$
\sum_{i=1}^{\infty} a_i T^{\lambda_i}e^{\mu_i}
$$
with $a_i \in \C$, $\lambda_i \in \R$, $\mu_i \in \Z$ such that
$\lambda_i \leq \lambda_{i+1}$ and $\lim_{i\to \infty} \lambda_i = \infty$,
unless it is a finite sum. Here $T$ and $e$ are formal parameters. We define
a valuation $\frak v_T: \Lambda_{\text{\rm nov}} \to \R$  by
$$
\frak v_T \left(\sum_{i=1}^\infty a_iT^{\lambda_i}e^{\mu_i}\right) = \lambda_1.
$$
This induces a natural $\R$-filtration on $\Lambda_{\text{\rm nov}}$
which in turn induces a non-Archimedean topology thereon.
Then we define
$\Lambda_{0, \text{\rm nov}}$ to be the subring of $\Lambda_{\text{\rm nov}}$ consisting of
$\sum a_i T^{\lambda_i}e^{\mu_i}$ with
$\frak v_T \left(\sum_{i=1}^\infty a_iT^{\lambda_i}e^{\mu_i}\right) \geq 0
$
and $\Lambda_{0,\text{\rm nov}}^+$ by the subring  with $\frak v_T> 0$.

We define $C(L^{(1)},L^{(0)};\Lambda_{\text{nov}})$ as the $\Lambda_{\text{nov}}$-module
generated by $\text{Crit }{\mathcal A}_{\ell_a}$, $a \in \pi_0(\Omega(L^{(0)},L^{(1)}))$
modulo the equivalence relation $\sim$ given in \eqref{eq:novikov}.
The filtration $\{F^{\lambda}\}$
on $C(L^{(1)},L^{(0)};\Lambda_{\text{nov}})$ is given by the action functional ${\mathcal A}_{\ell_a}$. See p.127 in \cite{fooo:book}.
We can regard $C(L^{(1)},L^{(0)};\Lambda_{\text{nov}})$ as a free
$\Lambda_{\text{nov}}$-module generated by $L^{(0)} \cap L^{(1)}$ provided
$L^{(0)}$ and $L^{(1)}$ intersect transversally.
In such a situation, we  can identify $F^{0}C(L^{(1)},L^{(0)};\Lambda_{\text{nov}})$ and
the free $\Lambda_{0,\text{nov}}$-module generated by $L^{(0)} \cap L^{(1)}$.
We defined a filtered $A_{\infty}$-bimodule structure on $C(L^{(1)},L^{(0)};\Lambda_{\text{nov}})$
in Theorem  3.7.21  in \cite{fooo:book} (see also Definition 3.7.41 in \cite{fooo:book}).
By extending the coefficient ring to $\Lambda_{\text{nov}}$, we also have
a filtered $A_{\infty}$-bimodule structure on  $C(L^{(1)},L^{(0)};\Lambda_{\text{nov}})$.
This construction does not rely on the choice of
the base paths $\ell_a$.
However, when we construct a filtered $A_{\infty}$-bimodule homomorphism
$C(L^{(1)},L^{(0)};\Lambda_{\text{nov}}) \to C(L^{(1) \prime},L^{(0) \prime};\Lambda_{\text{nov}})$,
we use the base baths $\ell_a$ and $\ell_a'$.
As we will see, the improved estimate in Section \ref{sec:improved} is used to
control the filtration change under the filtered $A_{\infty}$-bimodule homomorphism.

\subsection{Statement of Theorem J \cite{fooo:book}}
\label{subsec:theoremJ}

In \cite{fooo:book},
we associate a set $\MM_{\text{weak,def}}(L)$ for
each relatively spin Lagrangian submanifold $L$ of $(X,\omega)$ and the maps
$$
\pi_{\text{amb}} : \MM_{\text{weak,def}}(L)
\to
H^2(X;\Lambda_{0, \text{nov}}^+), \quad
\frak{PO} : \MM_{\text{weak,def}}(L)
\to \Lambda_{0, \text{nov}}^+
$$
such that
the Floer cohomology
$HF((L,\text{\bf b}_1),(L,\text{\bf b}_0);\Lambda_{0,\text{\rm nov}})$
can be defined whenever the following condition
holds:
$\MM_{\text{weak,def}}(L)\ne \emptyset$
and
$$
\pi_{\text{amb}}(\text{\bf b}_1) =
\pi_{\text{amb}}(\text{\bf b}_0),
\quad
\frak{PO} (\text{\bf b}_1) =
\frak{PO} (\text{\bf b}_0).
$$
See Theorem B \cite{fooo:book}.
When this condition is satisfied, we say $L$ is {\it weakly unobstructed after bulk deformation}.
We set
$$
\MM_{\text{weak}}(L)=\pi_{\text{amb}}^{-1}(0),
\quad
\MM (L) =\pi_{\text{amb}}^{-1}(0) \cap
\frak{PO}^{-1}(0),
$$
whose elements are called {\it weak bounding cochain} (weak Maurer-Cartan element),
{\it bounding cochain} (Maurer-Cartan element),
respectively.
See Section 3.6, especially Definition 3.6.4 and Definition
3.6.29 \cite{fooo:book} for the precise definitions of
bounding cochain and weak bounding cochain.
More generally, for a relative spin pair
$(L^{(1)},L^{(0)})$ of
Lagrangian submanifolds and
\bea\label{eq:MMfiber}
(\text{\bf b}_1, \text{\bf b}_0)
& \in
\{ (\text{\bf b}_1, \text{\bf b}_0) ~\vert~
\pi_{\text{amb}}(\text{\bf b}_1) =
\pi_{\text{amb}}(\text{\bf b}_0), ~
\frak{PO} (\text{\bf b}_1) =
\frak{PO} (\text{\bf b}_0)
\} \nonumber\\
& =: \MM_{\text{weak,def}}(L^{(1)})
\times_{(\pi_{\text{amb}},\frak{PO})}
\MM_{\text{weak,def}}(L^{(0)}),
\eea
we can define
the Floer cohomology $HF((L^{(1)},\text{\bf b}_1),(L^{(0)},\text{\bf b}_0);\Lambda_{0,\text{\rm nov}})$
over $\Lambda_{0,\text{\rm nov}}$.
By Theorem 6.1.20 \cite{fooo:book}, it is
isomorphic to
\be\label{eq:HF0nov} \Lambda_{0,\text{\rm nov}}^{\oplus a} \oplus
\bigoplus_{i=1}^b (\Lambda_{0,\text{\rm nov}}/T^{\lambda_i}\Lambda_{0,\text{\rm nov}})
\ee
for some non negative integer $a$ and positive real numbers $\lambda_i$ ($i=1,\dots,b$).
We call $a$ the {\it Betti number} and
$\lambda_i$ the {\it torsion exponent} of
the Floer cohomology.
We note that $HF((L^{(1)},\text{\bf b}_1),(L^{(0)},\text{\bf b}_0);\Lambda_{0,\text{\rm nov}})$
is not invariant under the Hamiltonian isotopy.
However, it is proved in \cite{fooo:book} (see Theorem G (G.4)) that the Floer cohomology
$$
HF((L^{(1)},\text{\bf b}_1),(L^{(0)},\text{\bf b}_0);\Lambda_{\text{\rm nov}})
$$
with
$\Lambda_{\text{\rm nov}}$ its coefficients is invariant
under the Hamiltonian isotopy and satisfies
\be\label{eq:HFLambda}
HF((L^{(1)},\text{\bf b}_1),(L^{(0)},\text{\bf b}_0);\Lambda_{\text{\rm nov}}) \cong
\Lambda_{\text{\rm nov}}^{\oplus a}.
\ee
In particular, when $a \neq 0$,
$L^{(0)}, \, L^{(1)}$ can not be displaced from each other. On the
other hand, when $a = 0$, there is no obvious obstruction to the
displacement. In this case, the torsion part of
$HF((L^{(1)},\text{\bf b}_1),(L^{(0)},\text{\bf b}_0);\Lambda_{0,\text{\rm nov}})$ provides
some information on the Hofer distance and
the displacement energy of the two.

Now, under the above brief review of
Lagrangian Floer theory for a weakly unobstructed
Lagrangian submanifold after bulk deformation,
we can state Theorem J of \cite{fooo:book}.

\medskip

\begin{thm}[Theorem J \cite{fooo:book}]\label{thm:correctedthmJ}
Let $(L^{(0)},L^{(1)})$ be a relatively spin pair of Lagrangian submanifolds
of $X$ and
$L^{(1)},L^{(0)}$ weakly unobstructed after bulk deformations. Let
$(\text{\bf b}_1,\text{\bf b}_0) \in \mathcal
M_{\text{\rm weak,def}}(L^{(1)}) \times_{\pi_{\text{\rm
amb}},\frak{PO}} \mathcal M_{\text{\rm weak,def}}(L^{(0)})$ as in \eqref{eq:MMfiber} and $\psi
: X \to X$ a Hamiltonian diffeomorphism.
Assume that $\psi(L^{(1)})$ is transversal to $L^{(0)}$ and denote
$$
b(\Vert\psi\Vert) = \#\{ i \mid\lambda_i \ge \Vert\psi\Vert\},
$$
where $\lambda_i$ are the torsion exponents as in \eqref{eq:HF0nov} and
$\Vert\phi\Vert$ is the Hofer norm defined by \eqref{def:Hofernorm}.
Then we have
\be\label{eq:mainformula} \# (\psi(L^{(1)})\cap L^{(0)}) \ge a + 2b(\Vert\psi\Vert).
\ee
\end{thm}

Theorem \ref{thm:correctedthmJ} follows from the following
Theorem 6.1.25 of \cite{fooo:book} (see Subsection 6.5.3 \cite{fooo:book}).
The proof of Theorem 6.1.25 contained an error which
we now fix.

We recall that a symplectic diffeomorphism
$\psi : (X,L) \to (X,L')$ induces a bijection
$$
\psi_{\ast} : \MM_{\text{weak,def}}(L) \to \MM_{\text{weak,def}}(L')
$$ which is compatible
with the maps $\pi_{\text{amb}}$ and $\frak{PO}$.
See Theorem B (B.3) \cite{fooo:book}.

\begin{thm}[Theorem 6.1.25 \cite{fooo:book}]\label{Theorem2}
Let $(L^{(0)}, L^{(1)})$ and
$(\text{\bf b}_1,\text{\bf b}_0)$
be as in Theorem \ref{thm:correctedthmJ}, and
$\psi^{(i)} : X \to X$ $(i=0,1)$ Hamiltonian diffeomorphisms. Put
$L^{(i)\prime}=\psi^{(i)}(L^{(i)})$.
Let $\lambda_{\downarrow,i}$, $i=1,\dots,b$ and $\lambda_{\downarrow,i}'$ $i=1,\dots,b'$ be the torsion exponents
of the Floer cohomology
$$
HF((L^{(1)},\text{\bf b}_1), (L^{(0)},\text{\bf b}_0);\Lambda_{0,\text{\rm nov}}),
\,~ HF((L^{(1)\prime},\psi^{(1)}_{\ast}\text{\bf b}_1), (L^{(0)\prime},
\psi^{(0)}_{\ast}\text{\bf b}_0);\Lambda_{0,\text{\rm nov}})
$$
respectively. We order them so that $\lambda_{\downarrow,i} \ge \lambda_{\downarrow,i+1}$
and
$\lambda_{\downarrow,i}' \ge \lambda_{\downarrow,i+1}'$.
Denote
\be\label{def:nu0}
\nu_0 = \dist(L^{(0)},L^{(0)\prime}) + \dist(L^{(1)},L^{(1)\prime}).
\ee
Then if $\lambda_{\downarrow,i} >  \nu_0$, then $i\le b'$, and
if $\lambda_{\downarrow,i} > \nu_0$ and $\lambda'_{\downarrow, i} > \nu_0$, then we have
\be\label{eq:Lipmainformula}
\mid \lambda_{\downarrow,i} - \lambda'_{\downarrow,i} \mid
\le \nu_0.
\ee
In particular, $\lambda_{\downarrow,i}$ is continuous for each
$i$ as long as $\lambda_{\downarrow,i} > 0$.
\end{thm}

\begin{rem}
Let $\lambda \in {\R}$ such that $\lambda > 2 \|H\|$.
In the statement (6.5.30) in p. 391 \cite{fooo:book}, we obtained
the chain maps
\bea
&{}&\phi:T^\lambda
C(L^{(1)},L^{(0)};\Lambda_{0,\text{\rm nov}}) \to
T^{\lambda-\|H\|}C(L^{(1)\prime},L^{(0)\prime});\Lambda_{0,\text{\rm nov}}))\label{eq:phi} \\
&{}&\phi':T^{\lambda-\|H\|}C(L^{(1)},L^{(0)};\Lambda_{0,\text{\rm nov}})
\to T^{\lambda-2\|H\|}C(L^{(1)},L^{(0)};\Lambda_{0,\text{\rm nov}})  \label{eq:phi'}.
\eea
The above mentioned error lies in the fact that the composition of \eqref{eq:phi} and \eqref{eq:phi'}
only chain homotopy equivalent to the inclusion
$$
\frak i : T^{\lambda}C(L^{(1)},L^{(0)};\Lambda_{0,\text{\rm nov}}) \longrightarrow
T^{\lambda-2\|H\|}C(L^{(1)},L^{(0)};\Lambda_{0,\text{\rm nov}})
$$
if we use the original energy estimate given in
Proposition 5.3.20 (Proposition 5.3.45)
\cite{fooo:book}.
Therefore we need to replace the rest of the proof by the following argument
which uses the construction of an optimal chain maps combining the coordinate
transformations explained in the previous sections and
the improved energy estimate.
\end{rem}

\subsection{Proof of Theorem 6.1.25 \cite{fooo:book}}
\label{subsec:6.1.25}

In this subsection we prove Theorem 6.1.25 \cite{fooo:book}.

Consider the pair $(\psi^{(0)},\psi^{(1)})$ of Hamiltonian diffeomorphisms.
As in \cite{fooo:book}, to simplify the notation,
we restrict ourselves to the case of a
transverse pair $(L^{(0)},L^{(1)})$ where both
$L^{(i)}$ are unobstructed, i.e., $\MM(L^{(i)}) \ne \emptyset$.
Then using bounding cochains $b_i \in \MM(L^{(i)})$,
we can define the coboundary operator $\delta_{b_1,b_0}$
on the filtered $A_{\infty}$ bimodule
$C(L^{(1)},L^{(0)};\Lambda_{0,\text{nov}})=F^0C(L^{(1)},L^{(0)};\Lambda_{\text{nov}})$.
(See Subsection 3.7.4 \cite{fooo:book}.)
Similarly, we have the coboundary operator
$\delta_{b_1^{\prime},b_0^{\prime}}$
on $C(L^{(1)\prime},L^{(0)\prime};\Lambda_{0,\text{nov}})=
F^0C(L^{(1)\prime},L^{(0)\prime};\Lambda_{\text{nov}})$, where
we put $b_i^{\prime}:=\psi^{(i)}_{\ast}b_i$.

Let $\delta > 0$ be given. We consider any Hamiltonian isotopy
$\phi_{H^{(0)}}, \, \phi_{H^{(1)}}$ generated by $H^{(0)}, \, H^{(1)}$
respectively such that $\phi_{H^{(i)}}^1 =\psi^{(i)}$,
\be\label{eq:Lprime} L^{(0)\prime} =
\psi^{(0)}(L^{(0)}), \quad L^{(1)\prime} =
\psi^{(1)}(L^{(1)})
\ee
and \be\label{eq:lengdelta} \leng(\phi_{H^{(0)}}) +
\leng(\psi_{H^{(1)}}) \leq \dist(L^{(0)}, L^{(0)\prime}) +
\dist(L^{(1)}, L^{(1)\prime}) + \delta. \ee Denote
$$
\nu_- = E^-(H^{(0)}) + E^+(H^{(1)}) ,
\quad \nu_+ =
E^+(H^{(0)}) + E^-(H^{(1)})
$$
and
$$
\nu : = \nu_- + \nu_+ = \|H^{(0)}\| + \|H^{(1)}\| = \leng(\phi_{H^{(0)}})
+\leng(\phi_{H^{(1)}}).
$$
We note that we can make $\nu$
as close to $\nu_0$ in \eqref{def:nu0}
as we want. See Remark \ref{rem:optimalestimate}.

We construct a filtered $A_{\infty}$ bimodule homomorphism
$$
\phi : C(L^{(1)},L^{(0)};\Lambda_{\text{nov}})
\to C(L^{(1)\prime},L^{(0)\prime};\Lambda_{\text{nov}}).
$$
One such construction is provided in in \cite{fooo:book}.
See (6.5.14) and (6.5.15) therein.
\par
However, we would like to have an additional property that is required in
Theorem \ref{Theorem2} above.
In \cite{fooo:book}, we used the moduli spaces of solutions for \eqref{eq:moving}, which is the
equation of moving Lagrangian boundary value problem.
In this article we use the moduli spaces of solutions
$u$ for \eqref{eq:CRJH}, instead of \eqref{eq:moving},
with $\rho=\rho_+$ such that
$$
\lim_{\tau \to -\infty} u(\tau, \cdot )=\ell_p,  \  \lim_{\tau \to +\infty} u(\tau,\cdot )=
{\mathfrak g}_{H^{(0)},H^{(1)}}^{+} (\ell_{q''}), \
\Phi_{H^{(0)},H^{(1)}} (w \#u) \sim w''
$$
as in Proposition \ref{prop:improvedestimate}.
Then by identifying
$\text{Crit}~{\mathcal A}_{\widehat{H},\ell_a}$ with
$\text{Crit}~{\mathcal A}_{\ell_a''}$
we obtain a filtered $A_{\infty}$ bimodule
homomorphism $\phi$ in a way similar to
Lemma 5.3.25 and Lemma 5.3.8 in \cite{fooo:book}.
The filtered $A_{\infty}$ bimodule homomorphism
induces a morphism of cochain complexes,
which we also denote by $\phi$ by an abuse of notation:
\be\label{eq:imporvedphi}
\phi: C(L^{(1)},L^{(0)};\Lambda_{\text{nov}}) \to C(L^{(1)\prime},L^{(0)\prime};\Lambda_{\text{nov}}).
\ee
Similarly, we use the moduli spaces of solutions $u$ for \eqref{eq:CRJH} with
$\rho=\rho_-$ such that
$$
\lim_{\tau \to -\infty} u(\tau, \cdot )=
{\mathfrak g}_{H^{(0)},H^{(1)}}^{+} (\ell_{p''}),
\  \lim_{\tau \to +\infty} u(\tau,\cdot )=\ell_{q}, \
\Phi_{H^{(0)},H^{(1)}} ^{-1}(w'') \#u \sim w,
$$
to obtain
\be \label{eq:improvedphi'}
\phi':C(L^{(1)\prime}, L^{(0)\prime};\Lambda_{\text{nov}}) \to
C(L^{(1)},L^{(0)};\Lambda_{\text{nov}}).
\ee
Since we choose $\ell_a$ in such a way that $c(\widehat{H},\ell_a)=0$
for all $a \in \pi_0(\Omega(L^{(0)},L^{(1)}))$,
Proposition \ref{prop:improvedestimate} implies that these composition satisfies
$$\phi: F^{\lambda} C(L^{(1)},L^{(0)};\Lambda_{\text{nov}}) \to
F^{\lambda - \nu_-} C(L^{(1) \prime},L^{(0) \prime};\Lambda_{\text{nov}}).$$
Similarly, we obtain
$$\phi':F^{\lambda} C(L^{(1) \prime},L^{(0) \prime};\Lambda_{\text{nov}}) \to
F^{\lambda - \nu_+} C(L^{(1)},L^{(0)};\Lambda_{\text{nov}}).$$
This leads to the chain map
\be\label{eq:phi'phi}
\phi'\circ \phi : F^{\lambda}C(L^{(1)},L^{(0)};\Lambda_{\text{\rm nov}}) \longrightarrow
F^{\lambda-\nu}C(L^{(1)},L^{(0)};\Lambda_{\text{\rm nov}}).
\ee
Equivalently, we can rewrite these into the chain maps
$$
T^{\nu_-}\phi
: F^{\lambda} C(L^{(1)}, L^{(0)};\Lambda_{\text{\rm nov}}) \to
F^{\lambda} C(L^{(1)\prime}, L^{(0)\prime};\Lambda_{\text{\rm nov}})
$$
and
$$
T^{\nu_+}\phi'
: F^{\lambda}C(L^{(1)\prime}, L^{(0)\prime};\Lambda_{\text{\rm nov}}) \to
F^{\lambda} C(L^{(1)}, L^{(0)};\Lambda_{\text{\rm nov}}).
$$
Setting $\lambda=0$, we have
$$
(T^{\nu_-}\phi)_* : HF((L^{(1)},b_1), (L^{(0)},b_0);\Lambda_{0,\text{\rm nov}})
\to
HF((L^{(1)\prime},b_1^{\prime}), (L^{(0)\prime}, b_0^{\prime});\Lambda_{0,\text{\rm nov}})
$$
and
$$
(T^{\nu_+}\phi')_* :
HF((L^{(1)\prime},b_1^{\prime}), (L^{(0)\prime}, b_0^{\prime});\Lambda_{0,\text{\rm nov}})
\to
HF((L^{(1)},b_1), (L^{(0)},b_0);\Lambda_{0,\text{\rm nov}})
$$
respectively.

We denote
$$
\frak i:F^{\lambda}C(L^{(1)},L^{(0)};\Lambda_{\text{\rm nov}})  \hookrightarrow
F^{\lambda-\nu}C(L^{(1)},L^{(0)};\Lambda_{\text{\rm nov}})
$$
the inclusion induced homomorphism.
\begin{lem}
The two maps
$$
(T^{\nu_+}\phi')\circ (T^{\nu_-}\phi), \, T^{\nu} \frak i:
F^{\lambda}C(L^{(1)},L^{(0)};\Lambda_{\text{\rm nov}}) \longrightarrow
F^{\lambda}C(L^{(1)},L^{(0)};\Lambda_{\text{\rm nov}})
$$
are chain homotopic to each other.
\end{lem}

\begin{proof}
This last statement follows from the arguments in p.390-391 \cite{fooo:book},
and also from the explicit energy bound \eqref {energy bound} in
Proposition \ref{prop:finalestimate} for solutions $u$ of \eqref{eq:CRJH} with
$\widehat{H}$ and $\rho= \rho_K$, $0 \leq K < \infty$, which are used to define the chain homotopy map.
Recall that $\rho_K$ extends smoothly to $K=0$ as the constant function zero.
Then the moduli spaces of solutions of \eqref{eq:CRJH} with $\rho=\rho_K$ in \eqref{rho_K}
defines a chain homotopy between $\phi' \circ \phi$ and the identity.

As for the filtrations, we apply \eqref{action bound} in Proposition \ref{prop:finalestimate} to
all the elements in the associated parameterized moduli space defining the
chain homotopy and find that the energy loss is bounded by $\nu$ for all $K$.
This proves that
$\phi' \circ \phi$ is chain homotopic to $\frak i$ as a map
$$
F^{\lambda}C(L^{(1)}, L^{(0)};\Lambda_{\text{\rm nov}}) \to
F^{\lambda-\nu}C(L^{(1)}, L^{(0)};\Lambda_{\text{\rm nov}}).
$$
\end{proof}

From now on, we consider the case that $\lambda=0$.
Then we have
\begin{equation}\label{eq:comphomot}
(T^{\nu_+}\phi')_*  \circ (T^{\nu_-}\phi)_* = T^{\nu}
\end{equation}
where
$$
T^{\nu} :
HF((L^{(1)},b_1), (L^{(0)},b_0);\Lambda_{0,\text{\rm nov}})
\to
HF((L^{(1)},b_1), (L^{(0)},b_0);\Lambda_{0,\text{\rm nov}})
$$
is the map $x \mapsto T^{\nu} x$.
\par
Since
$$
(T^{\nu_-}\phi)_*: HF((L^{(1)},b_1), (L^{(0)},b_0);\Lambda_{0,\text{\rm nov}}) \to
HF((L^{(1)\prime},b_1^{\prime}), (L^{(0)\prime},b_0^{\prime});\Lambda_{0,\text{\rm nov}})
$$
and
$$
(T^{\nu_+}\phi')_*: HF((L^{(1)\prime},b_1^{\prime}), (L^{(0)\prime},b_0^{\prime});\Lambda_{0,\text{\rm nov}}) \to
HF((L^{(1)},b_1),(L^{(0)},b_0);\Lambda_{0,\text{\rm nov}})
$$
are $\Lambda_{0,\text{\rm nov}}$-module
homomorphisms,
we have, for any $\lambda > 0$,
$$
(T^{\nu_-}\phi)_*
: T^{\lambda}HF((L^{(1)}b_1), (L^{(0)},b_0);\Lambda_{0,\text{\rm nov}}) \to
T^{\lambda}HF((L^{(1)\prime},b_1^{\prime}), (L^{(0)\prime}, b_0^{\prime});\Lambda_{0,\text{\rm nov}})
$$
and
$$
(T^{\nu_+}\phi')_*
: T^{\lambda}HF((L^{(1)\prime},b_1^{\prime}), (L^{(0)\prime}, b_0^{\prime});\Lambda_{0,\text{\rm nov}})
 \to
T^{\lambda}HF((L^{(1)}b_1), (L^{(0)},b_0);\Lambda_{0,\text{\rm nov}}).
$$
Since $(T^{\nu_+}\phi')_* \circ (T^{\nu_-} \phi)_*$ is equal to the multiplication by $T^{\nu}$,
the minimal number of generators of $\text{Im }(T^{\nu_+}\phi')_* \circ (T^{\nu_-} \phi)_*$ is
equal to $a+ b(\lambda + \nu)$ if $\lambda + \nu\notin \{\lambda_{\downarrow,i}\mid i= 1, \dots, b\}$.
On the other hand, the minimal number of generators of $T^{\lambda}HF((L^{(1)\prime},b_1^{\prime}),(L^{(0)\prime},b_0^{\prime});\Lambda_{0,\text{\rm nov}})$ is equal to $a+b'(\lambda)$ if $\lambda \notin \{\lambda'_{\downarrow,i}\mid i=1, \dots, b'\}$.
Here $$b'(\lambda)=\#\{i \mid \lambda'_{\downarrow,i} \geq \lambda\}.$$
Therefore we have
$$
a+b(\lambda + \nu) \geq a+ b'(\lambda)
$$
for
$\lambda \notin \{\lambda_{\downarrow,i} - \nu \mid i=1, \dots,b\} \cup \{\lambda'_{\downarrow,i} \mid i=1, \dots b'\}$ cf. Lemma 6.5.31 in \cite{fooo:book}.
This implies that $\lambda'_{\downarrow,i} \geq \lambda_{\downarrow,i} -\nu$ whenever
$\lambda_{\downarrow,i} > \nu$.

Since this holds for all Hamiltonian isotopies
$\phi_{H^{(0)}}, \, \phi_{H^{(1)}}$ satisfying \eqref{eq:Lprime}, \eqref{eq:lengdelta}
and for any $\delta > 0$, we obtain
\begin{equation}
\text{ if } \lambda_{\downarrow,i} > \nu, \quad \lambda_{\downarrow,i} \le \nu + \lambda'_{\downarrow,i}.
\end{equation}
By changing the role of $L^{(1)}, L^{(0)}$ with
$L^{(1) \prime}, L^{(0) \prime}$ we also obtain
\begin{equation}
\text{ if } \lambda'_{\downarrow,i} > \nu, \quad \lambda'_{\downarrow,i} \le \nu + \lambda_{\downarrow,i}.
\end{equation}
Theorem \ref{Theorem2} follows.
\qed

\begin{rem}\label{rem:optimalestimate} With given fixed
$L^{(0)\prime} \in \frak{Iso}(L^{(0)})$ and $L^{(1)\prime} \in \frak{Iso}(L^{(1)})$, we may consider
{\it all} possible Hamiltonian isotopy with given end points and take the infimum
of $\leng(\phi_{H^{(0)}}) + \leng(\phi_{H^{(1)}})$ over all $H^{(0)}$ and
$H^{(1)}$ such that
$$
\phi_{H^{(0)}}^1(L^{(0)}) = L^{(0)\prime}, \quad \phi_{H^{(1)}}^1(L^{(1)}) = L^{(1)\prime}.
$$
In this way, we can make $\leng(\phi_{H^{(0)}}) + \leng(\phi_{H^{(1)}})$
as close to the sum
$$
\dist(L^{(0)},L^{(0)\prime}) + \dist(L^{(1)},L^{(1)\prime})
$$
as we want.
\end{rem}

\section{Torsion threshold and displacement energy}\label{sec:torthre}
As we mentioned, the torsion exponents of the Floer
cohomology have some information on the displacement energy of Lagrangian submanifolds.
We introduce the following notion to describe
a relation between the torsion exponents and
the displacement energy.

\begin{defn}\label{defn:TLL'}
Let $L^{(1)}, L^{(0)}$ be weakly unobstructed Lagrangian submanifolds after bulk deformations.
Let
$$
(\text{\bf b}_1,\text{\bf b}_0) \in
\MM_{\text{weak,def}}(L^{(1)})
\times_{(\pi_{\text{amb}},\frak{PO})}
\MM_{\text{weak,def}}(L^{(0)})
$$
as in \eqref{eq:MMfiber}.
Suppose $HF((L^{(1)},\text{\bf b}_1),(L^{(0)},\text{\bf b}_0);\Lambda_{\text{\rm nov}}) = 0$.
We denote by $\lambda_i$ its torsion exponents
defined by \eqref{eq:HF0nov}.
\par
(1)
We define
$$
\frak
T((L^{(1)},\text{\bf b}_1),(L^{(0)},\text{\bf b}_0)) = \max_{i} \lambda_i
$$
and
call it the \emph{torsion threshold} of the pair $(L^{(0)},L^{(1)})$
relative to $(\text{\bf b}_1,\text{\bf b}_0)$.
\par
(2) We define
$$
\frak T(L^{(1)}, L^{(0)}) =
\sup_{(\text{\bf b}_1,\text{\bf b}_0)}
T((L^{(1)},\text{\bf b}_1),(L^{(0)},\text{\bf b}_0))
$$
and call it the
\emph{torsion threshold} of the pair $(L^{(0)},L^{(1)})$.
\par
When $HF((L^{(1)},\text{\bf b}_1),(L^{(0)},\text{\bf b}_0);\Lambda_{\text{\rm
nov}}) \neq 0$ for some
$(\text{\bf b}_1, \text{\bf b}_0)$, we define
$$\frak T(L^{(1)},L^{(0)}) = \infty.$$
\par
(3)
In the case $b_i \in \MM_{\text{weak}}(L^{(i)})$, we define $\frak{T}((L^{(1)},b_1),(L^{(0)},b_0))$ and
$\frak{T}(L^{(1)}, L^{(0)})$
in a similar manner. Here the supremum is taken over the set
$$
(b_1, b_0) \in
\{ (b_1, b_0) \in \MM_{\text{weak}}(L^{(1)})\times \MM_{\text{weak}}(L^{(0)}) ~\vert~
\frak{PO} (b_1) =
\frak{PO} (b_0)
\}.
$$
\par
(4) We just
denote $\frak T((L,\text{\bf b}),(L,\text{\bf b}))$ and
$\frak T(L,L)$
by $\frak T(L,\text{\bf b})$ and
$\frak T(L)$ respectively.
\end{defn}

\smallskip
We now specialize the energy estimate in the previous section
to the particular case
$$
(L^{(0)},L^{(1)}) = (L,L), \quad (L^{(0)\prime},L^{(1)\prime}) = (L,\psi^{(1)}(L))
$$
with the displacing condition
\be\label{eq:nointersect} L \cap
\psi^{(1)}(L) = \emptyset.
\ee
Then the following theorem relating the displacement energy and the
torsion threshold of $L$ is a special case of Theorem J. For readers' convenience, we
give its proof which specializes the proof of Theorem J to this particular context.

\begin{thm}\label{thm:eLTL}
Let $L$ be a relatively spin closed Lagrangian submanifold of $(X,\omega)$.
Suppose that $L$ is weakly unobstructed after bulk deformation and displaceable. We denote by
$e(L)(=e^X(L))$ its displacement energy.
Let the torsion threshold of
$HF((L, \text{\bf{b}}),(L,\text{\bf{b}});\Lambda_{0,\text{nov}})
$ be positive, i.e. assume $\frak T(L,\text{\bf{b}}) > 0$. Then we have
$$
e(L) \geq \frak T(L,\text{\bf{b}})
$$
for any $\text{\bf{b}} \in \MM_{\rm{weak,def}}(L)$. In particular, we have $e(L) \geq \frak T(L)$.
\end{thm}
\begin{proof}
Suppose to the contrary that there exist a sufficiently small $\delta > 0$ and an element
$\frak b \in \MM_{\rm{weak,def}}(L)$ such that
$$
e(L) < \frak T(L,\text{\bf{b}}) - \delta.
$$
Pick a Hamiltonian $H$ and its associated Hamiltonian isotopy $\phi_H$
such that
$$
\phi_H^1(L) \cap L = \emptyset, \quad \|H \| \leq e(L) + \delta.
$$
In particular, we also have
$$
\|H\| < \frak T(L,\text{\bf{b}}).
$$
Now we recall from \eqref{eq:phi'phi} that $\phi'_* \circ \phi_*$ restricts to
$$
T^\lambda HF((L,\text{\bf{b}}),(L,\text{\bf{b}});\Lambda_{0,\text{\rm nov}})
\to T^{\lambda - \|H\|} HF((L,\text{\bf{b}}),(L,\text{\bf{b}});\Lambda_{0,\text{\rm nov}}).
$$
and satisfies
$$
\phi'_* \circ \phi_* = (\phi'\circ \phi)_* = \frak i_*
$$
as a map
\be\label{eq:lambda-H}
T^\lambda HF((L,\text{\bf{b}}),(L,\text{\bf{b}});\Lambda_{0,\text{\rm nov}})
\to T^{\lambda - \|H\|} HF((L,\text{\bf{b}}),(L,\text{\bf{b}});\Lambda_{0,\text{\rm nov}})
\ee
for all $\lambda \in \R$.
\par
We now specialize to the case $\lambda= \|H\|$. In this case,
$$
T^{\lambda - \|H\|} HF((L,\text{\bf{b}}),(L,\text{\bf{b}});\Lambda_{0,\text{\rm nov}})
= HF((L,\text{\bf{b}}),(L,\text{\bf{b}});\Lambda_{0,\text{\rm nov}}).
$$
Since $\lambda < \frak T(L,\text{\bf{b}})$, the image of
the inclusion-induced map
$$
\frak i_*:T^\lambda HF((L,\text{\bf{b}}),(L,\text{\bf{b}});\Lambda_{0,\text{\rm nov}}) \to
HF((L,\text{\bf{b}}),(L,\text{\bf{b}});\Lambda_{0,\text{\rm nov}})
$$
is not trivial by the definition of $\frak T(L,\text{\bf{b}})$.

On the other hand, $HF((\phi^1_H(L), \phi^1_{H \ast}\text{\bf{b}}),(L,
\text{\bf{b}});\Lambda_{0,\text{\rm nov}}) = \{0\}$ by the hypothesis
$L \cap \phi^1_H(L) = \emptyset$ and hence $\phi_* = 0 =\phi'_*$ which implies
$\phi'_*\circ \phi_* = 0$.

Therefore the equality $\phi'_* \circ \phi_* = (\phi'\circ \phi)_* = \frak i_*$ with $\lambda = \|H\|$
in \eqref{eq:lambda-H} gives rise to a contradiction. This finishes the proof.
\end{proof}

\section{Displacement of polydisks inside cylinders in high dimensions}
\label{sec:polydisks}

In this section, we consider the situation of \cite{hind} in any dimension. Namely,
we prove Theorem \ref{thm:hindpoly} and
Theorem \ref{thm:hindbi} stated in Section \ref{sec:intro}.
\par
We recall the polydisks in $\C^n$ denoted by
$$
D(a_1,a_2,\ldots, a_n) = \{(z_1,\ldots, z_n) \in \C^n \mid \pi |z_1|^2 < a_1, \ldots,  \pi|z_n|^2 < a_n\}
$$
where $a_1 \leq a_2 \leq \cdots \leq a_n$. Hind considers only the case when $n = 2$. We also denote
the cylinder over the disk $|z_1|^2 \leq (a_1+\e)/\pi$ by
$$
Z_{1,n-1}(a_1+\e) = \{(z_1,\ldots, z_n) \mid \pi |z_1|^2 < a_1 + \e\}
$$
for $0< \e < 1$.

\begin{thm}[Theorem \ref{thm:hindpoly}]\label{thm:generalHind} Suppose that
$S>1$ and $0< \e <1$.
Let $Z_{1,n-1} = Z_{1,n-1}(1 + \e)$. Then we have
$$
S \leq e^{Z_{1,n-1}}(D(1,S,\ldots, S)).
$$
\end{thm}
\begin{proof} We prove this by contradiction. Suppose $e^{Z_{1,n-1}}(D(1,S,\cdots, S)) < S$ and so
$$
e^{Z_{1,n-1}}(D(1,S,\ldots, S)) < S - \delta
$$
for some small $\delta > 0$. By definition of
$e^{Z_{1,n-1}}(D(1,S,\ldots, S))$, there exists a Hamiltonian
$H$ on $Z_{1,n-1}$ such that
$$
\phi_H^1(D(1,S,\ldots, S)) \cap D(1,S,\ldots, S) = \emptyset, \quad \supp \phi_H \subset Z_{1,n-1}
$$
and
\be\label{eq:eZH}
\|H\| \leq e^{Z_{1,n-1}}(D(1,S,\ldots, S)) + \delta < S,
\ee
where the inequality comes from the standing hypothesis.
Since
$$\supp \phi_H \subset
\operatorname{Int} Z_{1,n-1}(1 + \e)
$$
is compact, we can symplectically embed
$$
D(1,S, \ldots, S) \subset S^2(1 +\e') \times \underbrace{S^2(\lambda) \times \cdots \times S^2(\lambda)}_{(n-1) \mbox{ times}}=:X
$$
together with the image of $D(1,S,\ldots ,S)$ by the isotopy
$\phi_H^t, 0\le t \le 1,$
for some $\e'$ with $0 < \e < \e'$ and sufficiently large $\lambda > 0$. For the later purpose, we take $\e'$ and $\lambda$
which satisfy $0 < \e < \e' <1$ and $\lambda >2S$.

We consider a circle $S^1(S) \subset S^2(\lambda)$ which divides $S^2(\lambda)$ into two domains of areas
$S$ and $\lambda -S$ respectively. Then we consider the torus
$$
L=L\left(\frac{1+\e'}{2},S,\ldots,S \right) =
S^1\left(\frac{1+\e'}{2}\right) \times S^1(S) \times \cdots \times S^1(S),
$$
which is a subset of $D(1,S, \ldots, S)$ because $\e'<1$.
This torus $L$ is displaceable by $\phi_H$
inside $X=S^2(1 +\e') \times \underbrace{S^2(\lambda) \times \cdots \times S^2(\lambda)}_{(n-1) \mbox{ times}}$
since $D(1,S, \ldots, S)$ is so. Therefore we have
$
e^X(L) \leq \|H\|
$
which follows from the definition of $e^X$.
In particular, by (\ref{eq:eZH}) we have
\be\label{eq:eMbetaT}
e^X(L) < S.
\ee
\par
On the other hand, we know that the torus
$$
L=L\left(\frac{1+\e'}{2},S,\ldots,S \right) \subset S^2(1 +\e') \times \underbrace{S^2(\lambda) \times \cdots \times S^2(\lambda)}_{(n-1) \mbox{ times}}
$$
is one of the toric fiber.
By Proposition 4.3 of
\cite{fooo:toric1} it
is weakly unobstructed
(i.e., $\MM_{\text{weak}}(L) \ne \emptyset$)
and we can
choose a weak bounding cochain $b \in \MM_{\text{weak}}(L) $ as $b = 0$.

Now it remains to show

\begin{lem} Choose the weak bounding cochain $b = 0$. Then
we have
$$
\frak T(L,0) \geq S.
$$
\end{lem}
\begin{proof}
By a result of \cite{cho-oh}
the Maslov index 2 disks are completely classified. They
consist of the obvious ones coming from the the upper and lower hemispheres of
$S^2(1+\e')$ which have equal areas $\frac{1 +\e'}{2}$, and those two
domains coming from $S^2(\lambda) \setminus S^1(S)$. The coboundary map $\frak m_1$ of
the Floer cochain complex of $L$
are contributed by these disks.
Since $\e'<1 <S$, the holomorphic disks with the minimal area are
the first two disks
$$
D^2_\pm\left(\frac{1+\e'}{2}\right) \times \{pt\} \times\cdots \times
\{pt\} \subset X,
$$
which cancel each other in the operation of $\frak m_1$.
See Case I-a in Subsection 3.7.6 \cite{fooo:book} and
Theorem 1.3 \cite {fooo:inv} for this cancellation argument.
The area of the next smallest
area disk is $S$ because $\lambda >2S$.
We have $(n-1)$ holomorphic disks with area $S$:
$$
\{pt\} \times \cdots  \times D_l^2(S) \times \{pt\} \times\cdots \times
\{pt\} \subset X, \quad l=2,\dots , n,
$$
where $D_l^2(S)$ is the disk with area $S$
bounding the circle
$S^1(S)$ in the $l$-th factor $S^2(\lambda)$
of $X$.
We note that
such holomorphic disks contribute to
$\frak m_1$ with the same sign. (See Theorem 11.1 (3) in
\cite{fooo:toric1} for more general result on orientations of moduli spaces of the Maslov index $2$ disks in toric manifolds.)
In particular, these holomorphic disks
do not cancel each other.
Then
the argument similar to one of Case I-b in Subsection 3.7.6
\cite{fooo:book} shows that
they produce a torsion part
$\Lambda_{0,{\text {nov}}}/T^S\Lambda_{0,{\text {nov}}}$
in the Floer cohomology of $L$.
It follows that
$$
\frak T(L,0) \geq S.
$$
\end{proof}

Combining \eqref{eq:eMbetaT} and this lemma, we obtain
$$
e^X(L) < \frak T(L,0).
$$
But this contradicts to Theorem \ref{thm:eLTL} and finishes the proof of
Theorem \ref{thm:generalHind}.
\end{proof}

By a similar argument, we can show the following variant of Theorem \ref{thm:generalHind}.
We consider the domain
$$
D_{n-k,k}(1,S): = D^2(1)^{n-k} \times B^{2k}(kS)
$$
for $k=1,\dots ,n-1$. Here $B^{2k}(kS)$ is the
ball in $\C^k$ of redius $r$ with the Gromov
width
$\pi r^2 = kS$.

\begin{thm}[Theorem \ref{thm:hindbi}] Suppose that
$S>1$ and $0< \e <1$.
Let $Z = Z_{n-k,k}(1 + \e) = D^2(1+\e)^{n-k} \times \C^k$.
Then we have
$$
S \leq e^{Z_{n-k,k}}(D_{n-k,k}(1,S)).
$$
\end{thm}
\begin{proof} The proof will be the same as that of Theorem \ref{thm:generalHind}
with the following modifications.
We again prove this by contradiction. Suppose
$
e^{Z_{n-k,k}}(D_{n-k,k}(1,S)) < S
$
and choose $\delta >0$ and $H$ as before so that
$$
e^{Z_{n-k,k}}(D_{n-k,k}(1,S)) < S - \delta
$$
and
$$
\phi_H^1(D_{n-k,k}(1,S)) \cap D_{n-k,k}(1,S) = \emptyset, \quad
\supp \phi_H \subset Z_{n-k,k}(1+\e),
$$
and
$$
\|H\| \leq e^{Z_{n-k,k}}(D_{n-k,k}(1,S)) + \delta < S.
$$
Then we can symplectically embed
$$
D_{n-k,k}(1,S) \subset S^2(1 +\e')^{n-k} \times  \C P^k(\lambda) =:X
$$
together with the image of $D_{n-k,k}(1,S)$ by the isotopy
$\phi_H^t, 0\le t \le 1,$ for some $\e'$ with $0 < \e < \e' <1$ and sufficiently large $\lambda > 0$.
Then we consider the torus
\beastar
L & = & S^1\left(\frac{1+\e'}{2}\right)^{n-k} \times S^1(S)^k  \\
& \subset &
S^2(1 +\e')^{n-k} \times B^{2k}(kS) \subset S^2(1 +\e')^{n-k} \times \C P^k(\lambda).
\eeastar
The torus $L$ is also contained in $D_{n-k,k}(1,S)$ because $\e'<1$.
Note that
$L$ is one of the toric fiber in $X=S^2(1 +\e')^{n-k} \times \C P^k(\lambda)$.
The rest of the proof is similar to
one of Theorem \ref{thm:generalHind}.
So we omit it.
\end{proof}

\bibliographystyle{amsalpha}

\end{document}